\documentclass[smallcondensed,envcountsect,envcountsame,amsart]{svjour3}     
\smartqed  
\usepackage{tikz}
\usepackage{tikz-cd}
\usepackage{amsmath}
\usepackage{amsfonts}
\usepackage{amssymb}

\usepackage{amsthm}
\usepackage{bm}
\smartqed
\usepackage{mathrsfs}
\usepackage{mathtools}
\usepackage{color}
\usepackage{graphicx}
\usepackage{comment}
\usepackage{graphicx}  
\usepackage[all]{xy}
\usepackage{amscd}

\newtheorem{thm}{Theorem}[section]   
\newtheorem{lem}[thm]{Lemma}          
\newtheorem{prop}[thm]{Proposition}
\newtheorem{axiom}{Axiom}

\newtheorem{defn}[thm]{Definition}

\newtheorem{cor}[thm]{Corollary}
\theoremstyle{definition}
\newtheorem{rem}[thm]{Remark}

\makeatother

\newcommand*{\longhookrightarrow}{\ensuremath{\lhook\joinrel\relbar\joinrel\rightarrow}}

\def\abb{{\mathbb{A}}}

\def\rbb{{\mathbb{R}}}

\def\ccal{{\mathcal{C}}}
\def\dcal{{\mathcal{D}}}

\def\ical{{\mathcal{I}}}
\def\jcal{{\mathcal{J}}}
\def\kcal{{\mathcal{K}}}

\def\qcal{{\mathcal{Q}}}
\def\rcal{{\mathcal{R}}}
\def\scal{{\mathcal{S}}}
\def\tcal{{\mathcal{T}}}

\def\czero{{\mathcal{C}^0}}

\def\coloneqq{\mathrel{\mathop:}=}
\numberwithin{equation}{section}

\begin{document}
\title{Model category of diffeological spaces}
\author{Hiroshi Kihara}
\institute{Hiroshi Kihara \at
	Center for Mathematical Sciences, University of Aizu, 
	Tsuruga, Ikki-machi, Aizu-Wakamatsu City, Fukushima, 965-8580, Japan \\
	Tel.: (+81)-242-37-2645\\
	Fax: (+81)-242-37-2752\\
	\email{kihara@u-aizu.ac.jp}           
}
\date{Received: date / Accepted: date}
\maketitle
\begin{abstract}
	The existence of a model structure on the category $\dcal$ of diffeological spaces is crucial to developing smooth homotopy theory. We construct a compactly generated model structure on the category $\dcal$ whose weak equivalences are just smooth maps inducing isomorphisms on smooth homotopy groups. The essential part of our construction of the model structure on $\dcal$ is to introduce diffeologies on the sets $\Delta^{p}$ $(p \geq 0)$ such that $\Delta^{p}$ contains the $k^{th}$ horn $\Lambda^{p}_{k}$ as a smooth deformation retract.
	
	
\end{abstract}
\section{Introduction}
The theory of model categories has developed explosively since its introduction by Quillen \cite{Q1} and has become the modern foundation of homotopy theory. Model categories have been constructed in many fields, such as algebraic geometry \cite{MV}, complex geometry \cite{L}, and operator algebra \cite{O}, to solve various problems using homotopical means.
\par\indent
In differential geometry, researchers are beginning to recognize the importance of finding a model category in which to do smooth homotopy theory (\cite{Wu}, \cite{CW}, \cite{HS}); in such a category, we can fully utilize a number of notions and consequences from the model category theory (cf. \cite{Hi}, \cite{Ho}). Since a model category must be complete and cocomplete by definition, we must construct a model structure on the category of some type of smooth spaces (i.e., generalized smooth manifolds).
\par\indent
The main objective of this paper is to endow the category $\dcal$ of diffeological spaces with a compactly generated model structure whose weak equivalences are just smooth maps inducing isomorphisms on smooth homotopy groups (see \cite[Definition 15.2.1]{MP} for a compactly generated model structure).
\par\indent
Since every object is fibrant with respect to our model structure (Theorem \ref{model}), cofibrant diffeological spaces play an important role in smooth homotopy theory as $CW$-complexes do in topological homotopy theory (cf. \cite[Section 1]{Quillenequiv}).
\par\indent
Construction of a model structure on the category $\dcal$ has been attempted by several authors. Since the standard $p$-simplex $\Delta^{p}$ endowed with the sub-diffeology of $\rbb^{p+1}$ does not contain the $k^{th}$ horn $\Lambda^{p}_{k}$ as a deformation retract, ingenuity is needed to construct a model structure on $\dcal$ in a manner similar to the case of topological spaces (\cite[Definition 7.10.6 and Example 11.1.8]{Hi}).
Wu \cite{Wu} and Christensen-Wu \cite{CW} defined fibrations, cofibrations, and weak equivalences, using affine $p$-spaces instead of the standard $p$-simplices (\cite[Definition 4.8]{CW}), and conjectured that with their definitions, $\dcal$ is a model category. In the latest version (version 6) of \cite{HS}, Haraguchi-Shimakawa 
used the notion of a tame map (\cite[Definition 3.8]{HS}) to define fibrations, and claimed that $\dcal$ is a model category which is not cofibrantly generated (\cite[Theorem 5.1]{HS}). 
However, there exists a gap in the proof of \cite[Lemma 5.6]{HS}.
\par\indent
We construct a compactly generated model structure on $\dcal$ by a new approach, as outlined in the following subsections.
\par\indent
In \cite{Quillenequiv}, we establish the Quillen equivalences between the model categories of diffeological spaces, simplicial sets, and arc-generated spaces, using adjoint pairs introduced in this paper. The results in this paper and \cite{Quillenequiv} are applied to the theory of finite- and infinite-dimensional $C^{\infty}$-manifolds in the succeeding papers.
\subsection{Category $\dcal$ of diffeological spaces.}
In this subsection, we recall the definition of a diffeological space and summarize the basic properties of the category $\dcal$ that are needed to state the main results.
\par\indent
A {\sl parametrization} of a set $X$ is a (set-theoretic) map $p: U \longrightarrow X$, where $U$ is an open subset of $\rbb^{n}$ for some $n$.
\begin{defn}\label{diffeosp}$(1)$ A {\sl diffeological space} is a set $X$ together with a specified set $D_X$ of parametrizations of $X$ satisfying the following conditions:
	\begin{itemize}
		\item[$\mathrm{(i)}$](Covering)  Every constant parametrization $p:U\longrightarrow X$ is in $D_X$.
		\item[$(\mathrm{ii})$](Locality) Let $p :U\longrightarrow X$ be a parametrization such that there exists an open cover $\{U_i\}$ of $U$ satisfying $p|_{U_i}\in D_X$. Then, $p$ is in $D_X$.
		\item[$(\mathrm{iii})$](Smooth compatibility) Let $p:U\longrightarrow X$ be in $D_X$. Then, for every $n \geq 0$, every open set $V$ of $\rbb^{n}$ and every smooth map $F  :V\longrightarrow U$, $p\circ F$ is in $D_X$.
	\end{itemize}
	The set $D_X$ is called the {\sl diffeology} of $X$, and its elements are called {\sl plots}.\\
	$(2)$ Let $X=(X,D_X)$ and $Y=(Y,D_Y)$ be diffeological spaces, and let $f  :X\longrightarrow Y$ be a (set-theoretic) map. We say that $f$ is {\sl smooth} if for any $p\in D_X$, \ $f\circ p\in D_Y$. Then, diffeological spaces and smooth maps form the category $\dcal$.
\end{defn}
The category $\dcal$ has the following properties (see Propositions \ref{category D} and \ref{adjointDC0}):
\begin{itemize}
	\item[(1)] The category $\mathcal{D}$ has initial and final structures with respect to the underlying set functor. In particular, $\mathcal{D}$ is complete and cocomplete. Further, the class of $\dcal$-embeddings (i.e., injective initial morphisms) is closed under pushouts and transfinite composites.
	\item[(2)] The category $\mathcal{D}$ is cartesian closed.
	\item[(3)] The underlying set functor $\dcal \longrightarrow Set$
	is factored as the underlying topological space functor 
	$\widetilde{\cdot}:\dcal \longrightarrow \czero$
	followed by the underlying set functor
	$\czero \longrightarrow Set$, where $\czero$ denotes the category of arc-generated spaces.
	Further, the functor
	$\widetilde{\cdot}:\dcal \longrightarrow \czero$
	has a right adjoint
	$R:\czero \longrightarrow \dcal$.
\end{itemize}
The notions of diffeological subspace and sub-diffeology are defined in the standard manner (see Definition \ref{embedding}).
\subsection{Standard $p$-simplices}
The principal part of our construction of a model structure on $\dcal$ is the construction of good diffeologies on the sets
$$
\Delta^p=\{(x_0,\ldots,x_p)\in\mathbb{R}^{p+1} \ |\ \underset{i}{\sum} x_i = 1,\ x_i\geq 0 \}\ \ \ (p\geq 0)
$$
which enable us to define weak equivalences, fibrations, and cofibrations and to verify the model axioms (cf. Section 1.3). The required properties of the diffeologies on $\Delta^{p} \ (p \geq 0)$ are expressed in the following four axioms:
\begin{axiom}
	The underlying topological space of $\Delta^p$ is the topological standard $p$-simplex for $p\geq 0$.
\end{axiom}
Recall that
$f:\Delta^p \longrightarrow \Delta^q$
is an {\sl affine map} if $f$ preserves convex combinations.
\begin{axiom}
	Any affine map $f:\Delta^p\longrightarrow \Delta^q$ is smooth.
\end{axiom}
Let $\scal$ denote the category of simplicial sets. For $K \in \scal$, the {\sl simplex category} $\Delta\downarrow K$ is defined to be the full subcategory of the overcategory $\scal\downarrow K$ consisting of maps $\sigma: \Delta[n] \longrightarrow K$ (\cite[p. 7]{GJ}). By Axiom 2, we can consider the diagram $\Delta\downarrow K \longrightarrow \mathcal{D}$ sending $\sigma :\Delta[n] \longrightarrow K$ to $\Delta^{n}$. Thus, we define the {\sl realization functor}
$$
|\ |_{\dcal}: \mathcal{S}\longrightarrow \mathcal{D}
$$
by $|K|_{\mathcal{D}}= \underset{\mathrm{\Delta\downarrow \text{\em K}}}{\mathrm{colim}} \ \Delta^n$.
\par\indent
Consider the smooth map $|\dot{\Delta}[p]|_{\dcal} \longhookrightarrow |\Delta[p]|_{\dcal} = \Delta^{p}$ induced by the inclusion of the boundary $\dot{\Delta}[p]$ into $\Delta[p]$.
\begin{axiom}
	The canonical smooth injection
	$$\left| \dot{\Delta}[p] \right|_{\dcal} \longhookrightarrow \Delta^p$$
	is a $\dcal$-embedding.
\end{axiom}

The notion of a deformation retract in $\dcal$ is defined in the same manner as in the category of topological spaces by using the unit interval $I=[0,\ 1]$ endowed with a diffeology via the canonical bijection with $\Delta^{1}$ (Section 2.4).

The {\sl $k^{th}$ horn} of $\Delta^p$ is a diffeological subspace of $\Delta^p$ defined by
\begin{eqnarray*}
	\Lambda^p_k &=&  \{(x_0,\ldots,x_p)\in\Delta^p \ |\ x_i=0 \hbox{ for some }i\neq k\}.
\end{eqnarray*}
\begin{axiom}
	The $k^{th}$ horn $\Lambda^p_k$ is a deformation retract of $\Delta^p$ in $\mathcal{D}$ for $p \geq 1$ and $0 \leq k \leq p$.
\end{axiom}
%
For a subset $A$ of the affine $p$-space $\abb^{p} = \{(x_0, \ldots, x_p) \in \rbb^{p+1} \ | \ \sum x_i = 1 \}$, $A_{\mathrm{sub}}$ denotes the set $A$ endowed with the sub-diffeology of $\abb^{p} \ (\cong \rbb^{p})$.
The set $\{\Delta^{p}_{\mathrm{sub}} \}_{p \geq 0}$ of diffeological spaces, used in \cite{H} to study diffeological spaces by homotopical means, is the first candidate of the standard $p$-simplices satisfying Axioms 1-4.
However, $\Delta^{p}_{\mathrm{sub}}$ satisfies neither Axiom 3 nor 4 for $p \geq 2$ (Proposition \ref{counterex}). Thus, we must construct a new diffeology on $\Delta^p$, at least for $p \geq 2$. Let us introduce such diffeologies on $\Delta^p$.
\par\indent
Let $(i)$ denote the vertex $(0, \ldots, \underset{(i)}{1}, \ldots, 0)$ of $\Delta^p$, and let $d^i$ denote the affine map from $\Delta^{p-1}$ to $\Delta^p$, defined by
\begin{eqnarray*}
	d^i((k))= \left \{
	\begin{array}{ll}
		(k) & \text{for} \ k<i,\\
		(k+1)& \text{for} \ k\geq i.
	\end{array}
	\right.
\end{eqnarray*} 
\begin{defn}\label{simplices}
	We define the {\sl standard $p$-simplices} $\Delta^p$ ($p\geq 0$) inductively. Set $\Delta^p=\Delta_{\mathrm{sub}}^p$ for $p\leq 1$. Suppose that the diffeologies on $\Delta^k$ ($k<p$) are defined. 
	We define the map 
	\begin{eqnarray*}
		\varphi_i: \Delta^{p-1}\times [0,1) & \longrightarrow & \Delta^p
	\end{eqnarray*}
	by $\varphi_{i}(x, t) = (1-t)(i)+td^{i}(x)$, and endow $\Delta^p$ with the final structure for the maps $\varphi_{0}, \ldots, \varphi_{p}$.
\end{defn}
Without explicit mention, the symbol $\Delta^p$ denotes the standard $p$-simplex defined in Definition \ref{simplices}; see Lemmas \ref{firstproperty} and \ref{secondproperty}, and Remarks \ref{skeleton} and \ref{fibrantapprox} for a comparison of the diffeologies of $\Delta^{p}$ and $\Delta^{p}_{\mathrm{sub}}$.
\par\indent
Since the diffeology of $\Delta^p$ is the sub-diffeology of $\abb^{p}$ for $p \leq 1$, our notion of a deformation retract in $\dcal$ coincides with the ordinary notion of a deformation retract in the theory of diffeological spaces (\cite[p. 110]{I} and Remark \ref{Dhomotopy}).
\par\indent
Axioms 1-4 for the standard $p$-simplices $\Delta^{p}$ are verified in Sections 3-8.
\subsection{Model structure on the category $\dcal$}
By Axiom 2, we can define the singular complex $S^{\dcal}X$ of a diffeological space $X$ to have smooth maps $\sigma : \Delta^p \longrightarrow X$ as $p$-simplices, thereby defining the {\sl singular functor} $S^{\mathcal{D}} :\mathcal{D} \longrightarrow \mathcal{S}$. We introduce a model structure on the category $\mathcal{D}$ of diffeological spaces in the following theorem; the proof is constructed using only properties (1)-(3) of the category $\dcal$ and Axioms 1-4 for the standard simplices. See \cite[Definition 15.2.1]{MP} for a compactly generated model category and \cite{K}, \cite{GJ}, or \cite{MP} for the model structure of $\mathcal{S}$.

\begin{thm}\label{model}
	Define a map $f :X\longrightarrow Y$ in $\mathcal{D}$ to be
	\begin{itemize}
		\item[$(1)$]
		a weak equivalence if $S^{\mathcal{D}} f:S^{\mathcal{D}} X\longrightarrow S^{\mathcal{D}} Y$ is a weak equivalence in the category of simplicial sets,
		\item[$(2)$]
		a fibration if the map $f$ has the right lifting property with respect to the inclusions $\Lambda^p_k \longhookrightarrow\Delta^p$ for all $p>0$ and $0\leq k\leq p$, and
		\item[$(3)$]
		a cofibration if the map $f$ has the left lifting property with respect to all maps that are both fibrations and weak equivalences.
	\end{itemize}
	With these choices, $\mathcal{D}$ is a compactly generated model category whose object is always fibrant.
\end{thm}
The following theorem shows that the singular complex $S^{\dcal}X$ captures smooth homotopical properties of $X$, and that our model structure on $\dcal$ organizes the smooth homotopy theory of diffeological spaces. See \cite[Section 3.1]{CW} or \cite[Chapter 5]{I} for the smooth homotopy groups $\pi^{\dcal}_{p}(X, x)$ of a pointed diffeological space $(X, x)$, and see \cite[p. 25]{GJ} for the homotopy groups $\pi_p(K, x)$ of a pointed Kan complex $(K, x)$.
\begin{thm}\label{homotopygp}
	Let $(X, x)$ be a pointed diffeological space. Then, there exists a natural bijection
	$$
	\varTheta_{X} : \pi^{\dcal}_{p}(X, x) \longrightarrow \pi_{p}(S^{\dcal}X, x) \ \  \text{for} \ \ p \geq 0,
	$$
	that is an isomorphism of groups for $p > 0$. \\
\end{thm}

From Theorem \ref{homotopygp}, we see that weak equivalences are just smooth maps inducing isomorphisms on smooth homotopy groups (Corollary \ref{characterization}).
\par\indent
Although we work with diffeological spaces in this paper, our construction of a model structure applies to the category of Chen spaces as well (\cite{BH}, \cite{Quillenequiv}).
\par\indent
The remainder of this paper is organized as follows. In Section 2, we review diffeological spaces and their underlying topological spaces. In Section 3, Axiom 1 for the standard $p$-simplices is verified. In Section 4, the diffeology of the standard $p$-simplex $\Delta^p$ is studied and the notion of a good neighborhood of an open simplex of $\Delta^p$, which is used to verify Axioms 2-4 for the standard $p$-simplices, is introduced. 
Axiom 2 is verified in Section 5. Using $\dcal$-homotopies constructed in Section 6, Axioms 3 and 4 are verified in Sections 7 and 8, respectively. Theorems \ref{model} and \ref{homotopygp} are proved in Section 9. In the Appendix, it is observed that $\Delta^{p}_{\mathrm{sub}}$ satisfies neither Axiom 3 nor 4 for $p > 1$.	

\section{Diffeological spaces and their underlying topological spaces}
In this section, we study diffeological spaces and their underlying topological spaces primarily from a categorical view point, and then introduce the basic homotopical notions in the category $\dcal$. Although we prove several useful lemmas in Section 2.1, the main objective of this section is to recall the basic definitions and results for the reader's convenience and to fix notation and terminology; most results in Sections 2.2 and 2.3 are found in \cite{CSW}, \cite{SYH}, and \cite{KM}. 
\par\indent
A good reference for diffeological spaces is the book \cite{I}; see also Section 2 of \cite{CSW}, which is an elegant three-page introduction to diffeological spaces. 
\subsection{Diffeological spaces}
The category $\dcal$ of diffeological spaces has the obvious underlying set functor. See \cite[pp. 230-233]{FK} for initial and final structures, and initial and final morphisms with respect to the underlying set functor. See \cite[Definition 14.1.6]{MP} for a transfinite composite.
\begin{prop}\label{category D}$(1)$ The category ${\dcal}$ has initial and final structures with respect to the underlying set functor. In particular, ${\dcal}$ is complete and cocomplete. Further, the class of $\dcal$-embeddings (i.e., injective initial morphisms) is closed under pushouts and transfinite composites.\\
	$(2)$ The category ${\dcal}$ is cartesian closed.
\end{prop}
\begin{proof}
	(1) Initial and final structures are constructed as in \cite[p. 90]{CSW}. Thus, it is easily seen that the underlying set functor $\dcal \longrightarrow Set$ creates limits and colimits (\cite[Theorem 2.5]{CSW}). To prove the last statement, suppose that the square
	\[
	\begin{tikzcd}
	A \arrow{r}{i} \arrow{d} & X \arrow{d} \\
	B \arrow{r}{j} & Y
	\end{tikzcd}
	\]
	is a pushout diagram in $\dcal$ such that $i$ is a $\dcal$-embedding. From the construction of colimits in $\dcal$, we see that $j$ is also a $\dcal$-embedding. Next suppose that $\lambda$ is an ordinal and $X: \lambda \longrightarrow \dcal$ is a $\lambda$-sequence in $\dcal$ (i.e., a colimit-preserving functor) such that the map $X_{\beta} \longrightarrow X_{\beta+1}$ is a $\dcal$-embedding for $\beta+1 < \lambda$. From the construction of colimits in $\dcal$, we can see that the transfinite composite $X_{0} \longrightarrow \underset{\beta < \lambda}{\mathrm{colim}} \ X_{\beta}$ is a $\dcal$-embedding by transfinite induction.\\
	(2) See \cite[pp. 35-36]{I}.
\end{proof}
By Proposition \ref{category D}(1) and \cite[Proposition 10.2.7]{Hi}, the class of $\dcal$-embeddings is closed under coproducts, and hence, under the formation of relative cell complexes (\cite[Definition 15.1.1]{MP}).
\par\indent
 Refer to \cite{BH} for further categorical properties of $\dcal$. Recall the following special cases of initial and final structures, which are most important in practice.
\begin{defn}\label{embedding}
	$(1)$ A smooth map $f :X \longrightarrow Y$ is called a {\sl $\dcal$-embedding} if $f$ is injective and initial with respect to the underlying set functor. A diffeological space $A$ is called a {\sl diffeological subspace} of $X$ if $A$ is a subset of $X$ endowed with the initial structure for the inclusion \ $A \longhookrightarrow X$; the diffeology of $A$ is called the {\sl sub-diffeology} of $X$.\\
	$(2)$ A smooth map $f  :X \longrightarrow Y$ is a {\sl $\dcal$-quotient map} if $f$ is surjective and final with respect to the underlying set functor. A diffeological space $Z$ is a {\sl quotient diffeological space} of $X$ if $Z$ is a quotient set of $X$ endowed with the final structure for the projection $X \longrightarrow Z$; the diffeology of $Z$ is the {\sl quotient diffeology} of $X$.
\end{defn}
The following three lemmas are used repeatedly in this paper.

\begin{lem}\label{initial}
	Consider the commutative square in $\dcal$
	$$
	\begin{tikzpicture}
	\draw[->] (-0.9,0.9) -- (0.9,0.9);
	\draw[->] (-1.1,0.7) -- (-1.1,-0.7);
	\draw[->] (-0.9,-0.9) -- (0.9, -0.9);
	\draw[->] ( 1.1,0.7) -- (1.1, -0.7);
	\node at(-1.11,0.9) {A};
	\node at(-1.11,-0.9) {B};
	\node at(1.11,0.9)  {X};
	\node at(1.11,-0.9) {Y};
	\node at(0,1.1) {i};
	\node at(0,-1.1) {j};
	\node at(-1.3,0) {p};
	\node at(1.3,0) {q};
	\node at(-6,0) {};
	\node at(6,0) {(S)};
	\end{tikzpicture}
	$$				
	with $j$ a $\dcal$-embedding. Then, the square $(S)$ is a pullback diagram in $\dcal$ if and only if $(S)$ is a pullback diagram in $Set$ and $i$ is a $\dcal$-embedding.						
\end{lem}
\begin{proof}$(\Longrightarrow)$ Since a pullback diagram in $\dcal$ is also a pullback diagram in $Set$ (see the proof of Proposition \ref{category D}(1)), we need to only show that $i$ is initial.
	\par\indent
	Let $S$ be a diffeological space and $f: S \longrightarrow A$ a set-theoretic map with $if$ smooth. Then, we have the implications
	\begin{center}
		$if$ is smooth $\Longrightarrow$ $qif$ is smooth $\Longrightarrow$ $pf$ is smooth, 
	\end{center}
	since $j$ is a $\dcal$-embedding. Since $(S)$ is a pullback diagram in $\dcal$, the smoothness of $if$ and $pf$ implies the smoothness of $f$.
	Hence, $i$ is initial.\\
	$(\Longleftarrow)$ Obvious.	
\end{proof}
\begin{lem}\label{initialfinal} 
	Suppose that
	$$
	\begin{tikzpicture}
	\draw[->] (-0.9,0.9) -- (0.9,0.9);
	\draw[->] (-1.1,0.7) -- (-1.1,-0.7);
	\draw[->] (-0.9,-0.9) -- (0.9, -0.9);
	\draw[->] ( 1.1,0.7) -- (1.1, -0.7);
	\node at(-1.11,0.9) {A};
	\node at(-1.11,-0.9) {B};
	\node at(1.11,0.9)  {X};
	\node at(1.11,-0.9) {Y};
	\node at(0,1.1) {i};
	\node at(0,-1.1) {j};
	\node at(-1.3,0) {p};
	\node at(1.3,0) {q};
	\end{tikzpicture}
	$$				
	is a pullback diagram in $\dcal$.\\
	$(1)$ If $j$ is a $\dcal$-embedding, then $i$ is also a $\dcal$-embedding.\\
	$(2)$ If $q$ is a $\dcal$-quotient map, then $p$ is also a $\dcal$-quotient map.
	\begin{proof}
		$(1)$ The result is immediate from Lemma \ref{initial}.\\
		$(2)$ Since $p$ is obviously surjective (see the proof of Proposition \ref{category D}(1)), we need to only show that $p$ is final. For a parametrization $\beta \ : U \longrightarrow B$, we have the implications
		\begin{align*}
		\beta \in D_B &\Longrightarrow j \circ \beta \in D_Y\\
		&\Longrightarrow  j \circ \beta \text{ lifts locally along $q$ in ${\dcal}$}\\
		&\Longrightarrow  \beta \text{ lifts locally along $p$ in ${\dcal}$},
		\end{align*}
		completing the proof (see \cite[p. 90]{CSW}).
	\end{proof}
\end{lem}
\begin{lem}\label{times X}
	Let $p: A \longrightarrow B$ be a $\dcal$-quotient map. Then, $p \times 1_X: A \times X \longrightarrow B \times X$ is also a $\dcal$-quotient map for any diffeological space $X$.
\end{lem}

\begin{proof}
	The result is easily proved using the cartesian closedness of $\dcal$ (Proposition \ref{category D}(2)). 
	Alternatively, since $p \times 1_X: A \times X \longrightarrow B \times X$ is a pullback of $p : A \longrightarrow B$ along $proj : B \times X \longrightarrow B$, the result can be deduced from Lemma \ref{initialfinal}(2). 
\end{proof}
\subsection{Arc-generated spaces}
A topological space $X$ is {\sl arc-generated} if the topology of $X$ is final for $C^{0}_X \ := \{$continuous curves from $\rbb$ to $X \}$ (cf. \cite[p. 45]{FK}). We investigate the fundamental properties of arc-generated spaces; arc-generated spaces form the target category of the underlying topological space functor for diffeological spaces (Section 2.3).
\par\indent
Let $\mathcal{T}$ be the category of topological spaces and $\mathcal{C}^0$ be the full subcategory of $\mathcal{T}$ consisting of arc-generated spaces.
\begin{prop}\label{categoryC0}
	The category $\mathcal{C}^0$ has initial and final structures with respect to the underlying set functor. In particular, $\mathcal{C}^0$ is complete and cocomplete.
\end{prop}


\begin{proof}
	Initial and final structures are easily constructed. See the proof of Proposition \ref{category D}(1) for the latter part.
\end{proof}
For a topological space $X$, $\alpha X$ is the set $X$ endowed with the final topology for $C^{0}_X$. Then, $\alpha$ defines a functor $ \ \alpha: \mathcal{T}\longrightarrow \mathcal{C}^0$. Let $I$ be the inclusion functor $\mathcal{C}^0 \longrightarrow \mathcal{T}$.
\begin{lem}\label{adjointC0T}
	$I  :\mathcal{C}^0 \rightleftarrows\mathcal{T}  :\alpha$ is an adjoint pair, and the composite $\alpha I$ is the identity functor of $\mathcal{C}^0$.
\end{lem}
\begin{proof}
	The result is obvious from the definitions.
\end{proof}
The following lemma shows that many important spaces are arc-generated.
\begin{lem}\label{exC0} $(1)$ If $X$ is a locally arcwise connected space that satisfies the first axiom of countability, then $X$ is arc-generated. \\
	$(2)$ Every $CW$-complex is arc-generated.
\end{lem}
\begin{proof} (1) The proof is not difficult (cf. the proof of \cite[Proposition 3.11]{CSW}).\\
	(2) Since a $CW$-complex is constructed from the disks $D^n$ and spheres $S^{n-1}$ via colimits in $\mathcal{T}$, the result follows from Part 1 and Lemma \ref{adjointC0T}.
\end{proof}
A topological space $X$ is {\sl $\Delta$-generated} (respectively, {\sl $\Delta^1$-generated}) if the topology of $X$ is final for the singular simplices (respectively, the singular 1-simplices). 
The notions of arc-generatedness, $\Delta^1$-generatedness, and $\Delta$-generatedness are mutually equivalent. This follows from the obvious implications
$$
\text{arc-generated} \ \Longrightarrow \ \Delta^{1}\text{-generated} \ \Longrightarrow \ \Delta\text{-generated},
$$
and the fact that $\Delta^p$ is arc-generated (Lemma \ref{exC0}(1)).
	\begin{prop}\label{convenC0}
		$(1)$ The category $\czero$ is cartesian closed. \\
		$(2)$ Let $X$ and $Y$ be arc-generated spaces, with $Y$ locally compact. Then, $X\underset{\mathcal{C}^0}{\times}Y= X \underset{\mathcal{T}}{\times}Y$ holds, where $\underset{{C}^0}{\times}$ and $\underset{\mathcal{T}}{\times}$ denote the products in $\mathcal{C}^0$ and $\mathcal{T}$, respectively.
		\begin{proof}
			Since being arc-generated is the same as being $\Delta^1$-generated, the result follows from \cite[Theorems 3.3 and 3.4, and Examples 3.5]{Wy}.
		\end{proof}	
	\end{prop}
	\subsection{Underlying topological space of a diffeological space}
	For a diffeological space $A = (A, D_A)$, the {\sl underlying topological space} $\widetilde{A}$ is defined to be the set $A$ endowed with the final topology for $D_{A}$; the  underlying topology of $A$ is also referred to as the {\sl $D$-topology} of $A$. (We often omit the symbol $\widetilde{\phantom A}$ if there is no confusion in context.)
	\par\indent
Since the $D$-topology of a diffeological space is arc-generated (Lemma \ref{exC0}(1)), we have the underlying space functor $\widetilde{\cdot}:\mathcal{D} \longrightarrow \czero$.
For an arc-generated space $X$, $RX$ is the set $X$ endowed with the diffeology $D_{RX}$ consisting of all continuous parametrizations. Then, $R$ defines a functor from $\mathcal{C}^0$ to $\mathcal{D}$.
\begin{prop}\label{adjointDC0}
	$\widetilde{\cdot}  :\mathcal{D} \rightleftarrows \czero  :R$ is an adjoint pair, and the composite $\, \widetilde{\cdot} \circ R$ is the identity functor of $\czero$.
\end{prop}
\begin{proof}
	Obvious.
\end{proof}
For the $D$-topology of a quotient diffeological space, we have the following result.
\begin{lem}\label{underlyingtop} The functors $\widetilde{\cdot}  :\mathcal{D} \longrightarrow \czero$ and $I  :\czero \longrightarrow \mathcal{T}$ preserve final structures with respect to the underlying set functors. In particular, if $Y$ is a quotient diffeological space of $X$, then $\widetilde{Y}$ is a quotient topological space of $\widetilde{X}$.
\end{lem}
\begin{proof}
	Since we have the two adjoint pairs
	$$\mathcal{D} \overset{\widetilde{\cdot}}{\underset{R}\rightleftarrows}\mathcal{C}^0\overset{I}{\underset{\alpha}\rightleftarrows}\mathcal{T}$$
	(Lemma \ref{adjointC0T} and Proposition \ref{adjointDC0}), the result follows immediately from \cite[Proposition 8.7.4]{FK}.
\end{proof}
The $D$-topology of a diffeological subspace need not coincide with the subspace topology (see \cite[Example 3.20]{CSW}). However, we have the following result, which is needed in this paper.
\begin{lem}\label{underlyingconvex}
	Let $V$ be a subset of $\rbb^{n}$. If $V$ is open or convex, then the $D$-topology of the diffeological subspace $V$ coincides with the topology induced from $\rbb^{n}$.		
\end{lem}
\begin{proof}
	We can easily prove the result using Special Curve Lemma \cite[p. 18]{KM} (cf. \cite[Section 3.3]{CSW}).
\end{proof}
Propositions \ref{convenC0}, \ref{adjointDC0}, and \ref{finiteproducts} explain that $\czero$ is more suitable than $\tcal$ as the target category of the underlying topological space functor for diffeological spaces.
\begin{prop}\label{finiteproducts}
	The underlying topological space functor $\widetilde{\cdot} :\mathcal{D} \longrightarrow \czero$ preserves finite products.
\end{prop}
\begin{proof}
	See \cite[Lemma 4.1]{CSW} and the comment after it.
\end{proof}

\subsection{Homotopical notions in the category $\mathcal{D}$}
Endow $I=[0,1]$ with a diffeology via the canonical bijection with $\Delta^1$. A {\sl $\mathcal{D}$-homotopy} (or {\sl smooth homotopy}) $H$ between smooth maps $f,g  :X\longrightarrow Y$ is a smooth map $H :X\times I\longrightarrow Y$ such that $H(\cdot, 0)=f$ and $H(\cdot,1)=g$.
\par\indent
A diffeological subspace $A$ of a diffeological space $X$ is a {\sl $\mathcal{D}$-deformation retract} (or {\sl deformation retract in $\mathcal{D}$}) if there is a $\mathcal{D}$-homotopy $H  :X\times I\longrightarrow X$ such that $H(x,0)=x$, $H(a,t)=a$, and $H(x,1)\in A$ for all $x\in X$, $a\in A$, and $t\in I$. Such a $\mathcal{D}$-homotopy is called a {\sl $\mathcal{D}$-deformation} of $X$ onto $A$.
\par\indent
We say that $f$ is {\sl $\mathcal{D}$-homotopic} to $g$ and write $f \simeq_{\dcal} g$ (or $f\simeq g$ in $\mathcal{D}$) if there is a $\mathcal{D}$-homotopy between $f$ and $g$. The relation $\simeq_{\dcal}$ is an equivalence relation (see Remark \ref{Dhomotopy}) and the quotient set $\mathcal{D}(X,Y)/{\simeq}_{\dcal}$ is denoted by $[X,Y]_{\mathcal{D}}$.
\par\indent
We say that diffeological spaces $X$ and $Y$ have the same {\sl $\dcal$-homotopy type} and write $X \simeq_{\dcal} Y$ (or $X \simeq Y$ in $\dcal$) if there exist smooth maps $f  : X \longrightarrow Y$ and $g  : Y \longrightarrow X$ such that $g \circ f \simeq_{\dcal} 1_X$ and $f \circ g \simeq_{\dcal} 1_{Y}$. A diffeological space $X$ is called {\sl $\dcal$-contractible} (or {\sl contractible in $\dcal$}) if $X$ has the same {\sl $\dcal$-homotopy type} as the terminal object $\ast$.
\begin{rem}\label{Dhomotopy}
	Since $\Delta^{1} = \Delta^{1}_{\mathrm{sub}}$ (Definition \ref{simplices}), the diffeology of $I$ is just the sub-diffeology of $\rbb$. Using a cut-off function, we can easily see that the relation $\simeq_{\dcal}$ is an equivalence relation and that our notion of smooth homotopy is equivalent to that of ordinary smooth homotopy (\cite[p. 1276]{CW}, \cite[p. 108]{I}).
\end{rem}

Let $I_{\text{top}}$ denote the unit interval $I=[0, 1]$ endowed with the ordinary topology. Since $I_{\text{top}}$ is arc-generated (Lemma \ref{exC0}), $\czero$-homotopies and the $\czero$-homotopy relation $\simeq_{\czero}$ are defined as in the category $\mathcal{T}$ of topological spaces. By Proposition \ref{convenC0}(2), $\czero$-homotopical notions are just the restrictions of ordinary homotopical notions.
\par\indent 
We have the following fundamental result.
\begin{prop}\label{DC0homotopies}
	Let $f,g :X\longrightarrow Y$ be smooth maps between diffeological spaces. If $f$ is $\mathcal{D}$-homotopic to $g$, then $\widetilde{f}$ is $\czero$-homotopic to $\widetilde{g}$. In particular, the natural map
	$$[X,Y]_{\mathcal{D}}\longrightarrow[\widetilde{X},\widetilde{Y}]_{\czero}$$
	exists.
	\begin{proof}
		Since $\widetilde{I} = I_{\mathrm{top}}$ (Lemma \ref{underlyingconvex}), the result is immediate from Proposition \ref{finiteproducts}.
	\end{proof}
\end{prop}
\section{Verification of Axiom 1}
In this section, we verify Axiom 1.
\par\indent
For a category $\ccal$ having the (faithful) underlying set functor, we adopt the following notation throughout this paper:
\begin{itemize}
	\item[(1)] For $X\in \ccal$, $1_X$ (or $1$) denotes the identity morphism of $X$.
	\item[(2)] For $X_1$, $X_2 \in \ccal$ having the same underlying set, $id : X_1 \longrightarrow X_2$ is the set-theoretic identity map or the morphism whose underlying set-theoretic map is the identity map.
\end{itemize}
\par\indent
Let us begin by proving the following lemma.
\begin{lem}\label{firstproperty}
	The identity map $id : \Delta^p\longrightarrow \Delta^p_{\mathrm{sub}}$ is smooth.
	\begin{proof}
		For $p\leq 1$, the result is obvious. Suppose that the result holds up to $p-1$. Consider the commutative diagram
		\begin{equation*}
		\begin{tikzcd}
		\Delta^{p-1}\times [0,1) \arrow{r}{id\times 1} \arrow{d}{\varphi_i} & \Delta^{p-1}_{\text{sub}}\times [0,1) \arrow{d}{\varphi_i}\\
		\Delta^p\arrow{r}{id} & \Delta^{p}_{\mathrm{\text{sub}}},
		\end{tikzcd}
		\end{equation*}
		and note that $\varphi_i\circ(id\times 1)$ is smooth by the induction hypothesis. Then, $id : \Delta^p\longrightarrow \Delta^p_{\text{sub}}$ is smooth by Definition \ref{simplices}.
	\end{proof}
\end{lem}

\begin{prop}[\bf Axiom 1]\label{Axiom 1}
	The underlying space of $\Delta^p$ is the topological standard $p$-simplex for $p\geq 0$.
	\begin{proof}
		The topological standard $p$-simplex is denoted by $\Delta^p_{\text{top}}$.
		\par\indent
		For $p=0$, the result is obvious.
		Suppose that $\widetilde{\Delta^{p-1}}=\Delta^{p-1}_{\text{top}}$. For a sufficiently small positive number $\epsilon$, consider the composite of smooth surjections
		$$
		\underset{i}{\textstyle \coprod}\ \Delta^{p-1}\times[0,1-\epsilon] \xrightarrow[ ]{\ \sum\varphi_i\ } \Delta^p \xrightarrow[ ]{\ id\ } \Delta^p_{\text{sub}}
		$$
		(see Lemma \ref{firstproperty}). Applying the functor $\widetilde{\cdot}$, we have the composite of continuous surjections
		$$
		\left(\underset{i}{\textstyle \coprod}\ \Delta^{p-1}\times[0,1-\epsilon]\right)^{\widetilde{}} \xrightarrow[ ]{\ \widetilde{\sum\varphi_i}\ } \widetilde{\Delta^p} \xrightarrow[ ]{\ id\ } \widetilde{\Delta^p_{\text{sub}}}.
		$$
		Observe that
		\begin{eqnarray*}
			\left(\underset{i}{\textstyle \coprod}\,\Delta^{p-1}\times[0,1-\epsilon]\right)^{\widetilde{}} 
			&=&\underset{i}{\textstyle \coprod}\,\Delta_{\text{top}}^{p-1}\underset{\mathcal{T}}{\times}[0,1-\epsilon]
		\end{eqnarray*}
		from Propositions \ref{adjointDC0} and \ref{finiteproducts}, the induction hypothesis, Lemma \ref{underlyingconvex}, and Proposition \ref{convenC0}(2). Then, $\left(\underset{i}{\coprod}\,\,\Delta^{p-1}\times[0,1-\epsilon]\right)^{\widetilde{}}$, and hence, $\widetilde{\Delta^{p}}$ is compact. Since the equality $\widetilde{\Delta^{p}_{\text{sub}}} = \Delta^{p}_{\mathrm{\text{top}}}$ holds (Lemma \ref{underlyingconvex}), $id : \widetilde{\Delta^{p}} \longrightarrow \widetilde{\Delta^{p}_{\text{sub}}} = \Delta^{p}_{\mathrm{\text{top}}}$ is a homeomorphism.
	\end{proof}\end{prop}
	\section{Local diffeological structure of the standard $p$-simplex $\Delta^p$}
	In this section, we study the diffeology of $\Delta^p$ and introduce the notion of a good neighborhood of an open simplex of $\Delta^p$, which is used to verify Axioms 2-4.
	\subsection{Comparison of the diffeologies of $\Delta^p$ and $\Delta^{p}_{\mathrm{sub}}$}
	We have seen that $id: \Delta^p \longrightarrow \Delta^{p}_{\mathrm{sub}}$ is smooth (Lemma \ref{firstproperty}). In this subsection, we further compare the diffeologies of $\Delta^p$ and $\Delta^{p}_{\mathrm{sub}}$. First, we introduce various important subsets of $\Delta^p$; the $k^{th}$ horn $\Lambda^p_k$ has already been introduced in Section 1.2.
	\begin{defn}\label{subsets}
		The {\sl interior} $\mathring{\Delta}^p$ and the {\sl boundary} $\dot{\Delta}^p$ of $\Delta^p$ are defined by 
		\begin{center}
			$\mathring{\Delta}^p = \{(x_0, \ldots, x_p) \in \Delta^p \ | \ x_i > 0 \ \ \text{for} \ \text{any} \ i \ \},$\\
			$\dot{\Delta}^p = \{(x_0,\ldots,x_p)\in\Delta^p\ |\ x_i=0 \hbox{ for some }i \}$.
		\end{center}
		For $0\leq i_1 < \cdots < i_k \leq p$, the {\sl closed $k$-simplex} $\langle i_0,\ldots,i_k \rangle $ and the {\sl open $k$-simplex} $(i_0,\ldots,i_k)$ are defined by
		\begin{eqnarray*}
			\langle i_0,\ldots, i_k \rangle & = & \{(x_0, \ldots, x_p) \in \Delta^p \ \mathrm{|} \ x_i=0 \ \ \text{for} \ \ i \ \notin \{i_0, \ldots, i_k \} \},\\
			(i_0,\ldots, i_k) & = & \{(x_0, \ldots, x_p) \in \Delta^p \ \mathrm{|} \ x_i=0 \ \ \text{if \ and \ only \ if} \ \ i \notin \{i_0, \ldots, i_k \} \}.
		\end{eqnarray*}
		Especially, $\langle 0, \ldots, \hat{i}, \ldots, \ p \rangle$ and $(0, \ldots, \hat{i}, \ldots, \ p)$ are often denoted by $\Delta^{p-1}_{(i)}$ and $\mathring{\Delta}^{p-1}_{(i)}$, respectively. The boundary $\dot{\Delta}^{p-1}_{(i)}$ of $\Delta^{p-1}_{(i)}$ is defined to be the subset $\Delta^{p-1}_{(i)}\backslash\mathring{\Delta}^{p-1}_{(i)}$.
		
		The subset $\Delta^p_{\hat{i}}  :=\Delta^p\backslash \Delta^{p-1}_{(i)}$ is called the {\sl $i^{th}$ half-open simplex}. The subset $\mathrm{sk}_l \ \Delta^p  :=\underset{k\leq l}{\bigcup}(i_0,\ldots,i_k)$ is called the {\sl $l$-skeleton} of $\Delta^p$.
			\end{defn}
		Without explicit mention, a subsets of $\Delta^p$ are endowed with the sub-diffeology of $\Delta^p$; a subset $A$ of $\Delta^p$ endowed with the sub-diffeology of $\Delta^p_{\mathrm{sub}}$ is denoted by $A_{\mathrm{sub}}$, which is compatible with the notation introduced in Section 1.2.

	\begin{lem}\label{secondproperty}		
		$\Delta^p-\mathrm{sk}_{p-2}\ \Delta^p=(\Delta^p-\mathrm{sk}_{p-2}\ \Delta^p)_{\mathrm{sub}}$.
		\begin{proof}
			For $p\leq 1$, the result is obvious.
			Suppose that the result holds up to $p-1$, and consider the following diagram:
			\begin{equation*}
			\begin{tikzcd}
			\underset{0\leq i\leq p}{\coprod}(\Delta^{p-1}-\mathrm{sk}_{p-3}\ \Delta^{p-1})\times (0,1) \arrow[hookrightarrow]{r} \arrow{d}[swap]{{\sum\varphi_i|_{\Delta^{p}-\mathrm{sk}_{p-2}\ \Delta^{p}}}}& \underset{0\leq i\leq p}{\coprod}\Delta^{p-1}\times [0,1) \arrow{d}{\sum\varphi_i}\\
			\Delta^{p}-\mathrm{sk}_{p-2}\ \Delta^{p}\arrow[hookrightarrow]{r}&\Delta^{p}.
			\end{tikzcd}
			\end{equation*}
			This diagram is a pullback diagram in $Set$. If $A \longhookrightarrow X$ and $B \longhookrightarrow Y$ are $\dcal$-embeddings, then $A \times B \longhookrightarrow X \times Y$ is also a $\dcal$-embedding. Thus, the upper horizontal arrow of the above diagram is a $\dcal$-embedding, and hence, the above diagram is a pullback diagram in $\dcal$ (Lemma \ref{initial}).
			Since $\sum\varphi_i$ is a $\dcal$-quotient map, $\sum\varphi_i|_{\Delta^{p}-\mathrm{sk}_{p-2}\ \Delta^{p}}$ is also a $\dcal$-quotient map (Lemma \ref{initialfinal}(2)). Next, consider the commutative diagram in $\dcal$:
			\begin{equation*}
			\begin{tikzcd}
			\underset{0\leq i\leq p}{\coprod}(\Delta^{p-1}-\mathrm{sk}_{p-3}\ \Delta^{p-1})\times (0,1) \arrow{r}{id} \arrow{d}[swap]{{\sum\varphi_i|_{\Delta^{p}-\mathrm{sk}_{p-2}\ \Delta^{p}}}}& \underset{0\leq i\leq p}{\coprod}(\Delta^{p-1}-\mathrm{sk}_{p-3}\Delta^{p-1})_{\mathrm{sub}} \times (0,1) \arrow{d}{\sum\varphi_i|_{\Delta^{p}-\mathrm{sk}_{p-2}\ \Delta^{p}}}\\
			\Delta^{p}-\mathrm{sk}_{p-2}\ \Delta^{p}\arrow{r}{id}&(\Delta^{p}-\mathrm{sk}_{p-2}\Delta^p)_{\mathrm{sub}}
			\end{tikzcd}
			\end{equation*}
			(see Lemma \ref{firstproperty}). We have shown that the left vertical arrow is a $\dcal$-quotient map. Since $\varphi_{i} \ |_{\Delta^{p}-\mathrm{sk}_{p-2} \ \Delta^{p}}  :(\Delta^{p-1}-\mathrm{sk}_{p-3}\Delta^{p-1})_{\mathrm{sub}} \times (0, 1) \longrightarrow (\Delta^{p}-\mathrm{sk}_{p-2}\Delta^{p})_{\mathrm{sub}}$ is obviously an open $\dcal$-embedding for any $i$, the right vertical arrow is also a $\dcal$-quotient map. Since the upper horizontal arrow is a diffeomorphism by the induction hypothesis, the lower horizontal arrow is also a diffeomorphism.
		\end{proof}
	\end{lem}
	The diffeologies of $\Delta^{p}$ and $\Delta^{p}_{\mathrm{sub}}$ actually differ along $\text{sk}_{p-2} \ \Delta^{p}$ (see Remark \ref{skeleton}).
	\subsection{Good neighborhood of an open simplex of $\Delta^p$}
	In view of Lemma \ref{secondproperty}, it is important to understand the diffeology of a neighborhood of each open simplex of $\Delta^p$ (especially of codimension $\geq 2$). This is the objective in this subsection.
	\par\indent 
	Let $I=\{i_0,\ldots,i_k\}$ be a subset of the ordered set $\{0, \ldots, \ p\}$; we always assume that the elements $i_0, \ldots,i_k$ are arranged such that $i_0 < \cdots <i_k$. Then, $\Delta^I$ is the closed $k$-simplex $\langle i_0,\ldots,i_k \rangle$ endowed with the diffeology that makes the affine bijection $\phi_I  :\Delta^k\longrightarrow \langle i_0,\ldots,i_k \rangle$ defined by $\phi_I((j))=(i_j)$ a diffeomorphism. The $i^{th}$ half-open simplex $\Delta^{I}_{\hat{i}}$ $(i\in I)$, the interior $\mathring{\Delta}^I$, and the boundary $\dot{\Delta}^I$ of $\Delta^I$ are defined in the obvious manner (cf. Definition \ref{subsets}). For $l\in I$,
	$$\varphi_l:\Delta^{I-\{l \}} \times[0,1)\longrightarrow\Delta^{I}$$
	is defined by $\varphi_l(x,t)=(1-t)(l)+tx.$ The map $\varphi_i : \Delta^{p-1} \times [0, 1) \longrightarrow \Delta^p$ in Definition \ref{simplices} is identified with $\varphi_i  : \Delta^{\{0, \ldots, \hat{i}, \ldots, p \}} \times [0, 1) \longrightarrow \Delta^p$ in the obvious manner. 
	\begin{lem}\label{halfopen}
		For $p > 0$ and $0 \leq j \leq p$, $\varphi_j :\Delta^{p-1}\times[0,1)\longrightarrow \Delta^{p}_{\hat{j}}$ is a $\dcal$-quotient map.
		\begin{proof}
			For $p=1$, the result is obvious.
			Suppose that the result holds up to $p-1$, and consider the following diagram: 
			\begin{equation*}
			\begin{tikzcd}
			\underset{i \neq j}{\coprod}\Delta^{\{0, \ldots, \hat{i}, \ldots, p\}}_{\hat{j}} \times (0, 1)\coprod\Delta^{\{0, \ldots, \hat{j}, \ldots, p\}} \times [0, 1) \arrow[hookrightarrow]{r} \arrow{d}[swap]{\underset{i \neq j}{\sum}\varphi_i + \varphi_j} & \underset{i}{\coprod}\ \Delta^{\{0, \ldots, \hat{i}, \ldots, p\}}\times [0,1) \arrow{d}{\sum\varphi_i}\\
			\Delta^{p}_{\hat{j}}\arrow[hookrightarrow]{r}&\Delta^{p}.
			\end{tikzcd}
			\end{equation*}
			We can show that this diagram is a pullback diagram in $\dcal$ by an argument similar to that in the proof of Lemma \ref{secondproperty}.
			Since $\underset{i}{\sum}\varphi_{i}$ is a $\dcal$-quotient map (Definition \ref{simplices}), $\underset{i \neq j}{\sum} \varphi_i + \varphi_j$
			is also a $\dcal$-quotient map (Lemma \ref{initialfinal}(2)). Thus, from the construction of a final structure (\cite[p. 90]{CSW}), we need to only show that for any $i$ \ $(\neq \ j),$ the map $\phi$ defined by the commutative diagram 
			\begin{equation*}
			\begin{tikzcd}
			\Delta^{\{0,\ldots,\hat{i},\ldots,p\}}_{\hat{j}}\times (0,1) \arrow{rr}{\phi} \arrow[swap]{dr}{\varphi_i} & &\Delta^{\{0,\ldots,\hat{j},\ldots,p\}}_{\hat{i}}\times (0,1)\arrow{dl}{\varphi_j}\\
			& \Delta^{p}_{\hat{j}} & 
			\end{tikzcd}
			\end{equation*}
			is smooth. 
			We easily see that
			$$
			\phi((1-\tau)(j)+\tau y,t)=\left( \frac{1-t}{1-t(1-\tau)}(i)+\frac{t\tau}{1-t(1-\tau)}y,1-t(1-\tau) \right).
			$$
			Thus we have the commutative diagram
			\begin{center}
				\begin{tikzcd}
					\Delta^{\{0,\ldots,\hat{i},\ldots,\hat{j},\ldots,p\}}\times [0,1)\times (0,1) \arrow{r}{\phi'} \arrow[swap]{d}{\varphi_j\times 1} & \Delta^{\{0,\ldots,\hat{i},\ldots,\hat{j},\ldots,p\}}\times [0,1)\times (0,1)\arrow{d}{\varphi_i\times 1}\\
					\Delta^{\{0,\ldots,\hat{i},\ldots,p\}}_{\hat{j}}\times (0,1)\arrow{r}{\phi} & \Delta^{\{0,\ldots,\hat{j},\ldots,p\}}_{\hat{i}}\times (0,1),
				\end{tikzcd}
			\end{center}
			where $\phi'$ is defined by
			$$
			\phi'(y,\tau,t)=\left(y,\frac{t\tau}{1-t(1-\tau)}, 1-t(1-\tau) \right).
			$$
			Note that $(\varphi_i\times 1)\circ \phi'$ is smooth, and that $\varphi_j\times 1$ is final by the induction hypothesis and Lemma \ref{times X}. Thus, $\phi$ is smooth.
		\end{proof}
	\end{lem}
	We also use the following lemma.
	\begin{lem}\label{permutation}
		Any permutation $\sigma$ of the vertices of $\Delta^p$ extends to a smooth affine map from $\Delta^p$ to $\Delta^p$.
		\begin{proof}
			Since $\underset{i}{\sum}\ \varphi_i  :\underset{i}{\coprod}\ \Delta^{p-1}\times[0,1)\longrightarrow\Delta^{p}$ is final (Definition \ref{simplices}), we can prove the result by induction.                        
		\end{proof}
	\end{lem}
	The following proposition describes the local diffeological structure of $\Delta^p$.
	\begin{prop}\label{goodnbd}
		Let $I = \{i_0, \ldots, i_k \}$ be a subset of the ordered set $\{0,\ldots, p \}$ with $0 \leq k < p$, and define the open neighborhood $U_I$ of the open simplex $(i_0, \ldots, i_k)$ of $\Delta^p$ by
		$$
		U_{I} = \{(x_{0}, \ldots, x_{p}) \in \Delta^p \ | \ x_{i} > 0 \ \text{for} \ i \in I \}.
		$$
		Then, the map
		$$
		\varPhi_{I}: U_I \longrightarrow \mathring{\Delta}^{k} \times \Delta^{p-k}_{\hat{0}}
		$$
		defined by
		$$
		\varPhi_{I} (x_0, \ldots, x_p) = ( \frac{x_{i_0}(0)+ \cdots + x_{i_k}(k)}{x_{i_0}+ \cdots + x_{i_k}}, (x_{i_0}+ \cdots + x_{i_k})(0)+x_{j_1}(1)+ \cdots +x_{j_{p-k}}(p-k))
		$$
		is a diffeomorphism, where $\{j_1, \ldots, j_{p-k} \} = \{0, \ldots, p \} \backslash \{i_0, \ldots, i_k \}$.
	\end{prop}
	For the proof of Proposition \ref{goodnbd}, we introduce the notion of the join of diffeological spaces. The {\sl join}, $X\ast Y$, of two diffeological spaces $X$ and $Y$ is defined to be the quotient diffeological space $X\times[0,1]\times Y/\sim$, where $(x,0,y)\sim(x,0,y')$ and $(x,1,y)\sim(x',1,y)$ for every $x,x'\in X$ and $y,y'\in Y$.
	\par\indent 
	Let $I=\{i_0, \ldots,i_k\}$ be a subset of the ordered set $\{0, \ldots, \ p\}$, and set $J=\{0,\ldots,p\}\backslash I$. Then, the canonical homeomorphism
	$$\psi_I:\Delta^I\ast\Delta^J\longrightarrow\Delta^p$$
	is defined by $\psi_I([x,\tau,y])=(1-\tau)x+\tau y$ (see Lemma \ref{underlyingtop} and Propositions \ref{finiteproducts}, \ref{Axiom 1}, and \ref{convenC0}). 
	It is an essential part of the proof of Proposition \ref{goodnbd} to show that $\psi_{I}$ restricts to a diffeomorphism from $\psi^{-1}_{I}(U_{I})$ to $U_{I}$. (Note that we do not insist that $\psi_{I}$ is smooth.)
	\begin{proof}[Proof of Proposition \ref{goodnbd}]
		We prove the proposition in three steps. \vspace{0.2cm} \\
		Step 1. We show that the diffeological subspace $\psi_{I}^{-1}(U_{I})$ of $\Delta^{I} \ast \Delta^{J}$ is diffeomorphic to $\mathring{\Delta}^{k} \times \Delta^{p-k}_{\hat{0}}$. Consider the diagram
		\begin{equation*}
		\begin{tikzcd}
		\mathring{\Delta}^{I} \times [0,1) \times \Delta^{J}
		\arrow[hook]{r} \arrow[swap]{d}{\pi|_{\psi^{-1}_{I}(U_{I})}} 
		& \Delta^{I} \times [0, 1] \times \Delta^{J} \arrow{d}{\pi}\\
		\psi^{-1}_{I}(U_{I}) \arrow[hook]{r}
		& \Delta^{I} \ast \Delta^{J},
		\end{tikzcd}
		\end{equation*}
		where $J = \{j_1, \ldots, j_{p-k}\} \ \text{and} \ \pi$ is the canonical $\dcal$-quotient map. We can show that this diagram is a pullback diagram in $\dcal$ by an argument similar to that in the proof of Lemma \ref{secondproperty}. Since $\pi|_{\psi^{-1}_{I}(U_{I})}$ is a $\mathcal{D}$-quotient map by Lemma \ref{initialfinal}(2),
		\begin{eqnarray*}
			\psi^{-1}_{I}(U_{I}) & = & \mathring{\Delta}^I\times [0,1)\times \Delta^{J}/\sim\\
			& \cong & \mathring{\Delta}^I\times ([0,1)\times \Delta^{J}/\sim)\\
			& \cong &  \mathring{\Delta}^k\times \Delta^{p-k}_{\hat{0}},
		\end{eqnarray*}
		where the equivalence relation $\sim$ on $[0, 1) \times \Delta^J$ is defined by $(0, y)\sim(0, y')$ for every $y, y' \in Y$ (see Lemmas \ref{times X} and \ref{halfopen}).\vspace{0.2cm} \\
		Step 2. We show that the homeomorphism 
		\[\psi^{-1}_{I}(U_{I}) = \mathring{\Delta}^{I} \times [0, 1) \times \Delta^{J}/\sim \ \xrightarrow{\ \psi_{I}\ } U_{I}
		\] 
		is a diffeomorphism.
		For $k=0$, $\psi_I$ is a diffeomorphism by Lemma \ref{halfopen}. Suppose that for any $I = \{i_0, \ldots, i_{k-1} \}$, $\psi_I$ is a diffeomorphism. We show that for any $I = \{i_0, \ldots, i_k \}$, $\psi_I$ is a diffeomorphism. For simplicity, let $I=\{0,\ldots,k\}$ (Lemma \ref{permutation}). For $I' \coloneqq\{1,\ldots,k\}$, the open neighborhood $U_{I'}$ of the open simplex $(1, \ldots, k)$ of $\Delta^{\{1,\ldots,p\}}$ is defined similarly. Then, the map
		$$
		\psi_{I'}:\mathring{\Delta}^{I'}\times[0,1)\times\Delta^J/\sim\longrightarrow U_{I'},
		$$
		defined by $\psi_{I'}([x', \sigma, y]) = (1-\sigma)x' + \sigma y$, is a diffeomorphism by the induction hypothesis, where $J=\{k+1, \ldots, \ p\}$. 
		
		Consider the solid arrow diagram
		\begin{equation*}
		\begin{tikzcd}[column sep=4em]
		(0,1)\times (\mathring{\Delta}^{I'}\times [0,1)\times \Delta^{J}/\sim)
		\arrow{r}{1\times\psi_{I'}} \arrow[swap]{r}{\cong} \arrow{d} 
		& (0,1)\times U_{I'} \arrow{d}\\
		\mathring{\Delta}^I\times [0,1)\times \Delta^{J}/\sim \arrow[dashrightarrow]{r}{\psi'}
		& U_I,
		\end{tikzcd}
		\end{equation*}
		where the left and right vertical arrows are defined by
		\[
		\begin{tikzcd}
		(s,[x',\sigma, y]) \arrow[mapsto]{d} \\
		{[}(1-s)(0)+sx', \sigma, y{]}
		\end{tikzcd}
		\text{and}
		\begin{tikzcd}
		(s,(1-\sigma)x'+\sigma y)	\arrow[mapsto]{d} \\
		(1-s)(0)+s((1-\sigma)x'+\sigma y),
		\end{tikzcd}
		\]
		respectively. Since the canonical isomorphisms
		$$(0, 1) \times (\mathring{\Delta}^{I'} \times [0,1) \times \Delta^{J}/\sim) \cong ((0,1)\times \mathring{\Delta}^{I'}) \times ([0,1) \times \Delta^{J}/\sim),$$
		$$ \mathring{\Delta}^{I} \times [0,1) \times \Delta^{J}/\sim \ \cong \mathring{\Delta}^I \times ([0,1) \times \Delta^{J}/\sim)$$
		exist (Lemma \ref{times X}), the left vertical arrow is obviously a diffeomorphism (see Lemma \ref{secondproperty}). The right vertical arrow is also a diffeomorphism since it is a restriction of the $\dcal$-quotient map
		$$
		[0,1)\times \Delta^{\{1, \ldots, p\}} \cong \Delta^{\{1, \ldots, p\}} \times [0, 1) \xrightarrow{\varphi_0} \Delta^{p}_{\hat{0}}
		$$
		(see Lemma \ref{halfopen}). Thus, we define the diffeomorphism
		$\psi': \mathring{\Delta}^I \times [0, 1) \times \Delta^J/\sim \longrightarrow U_I$
		by the commutativity of the above diagram. Then, we have only to show that the composite $\psi^{' -1} \circ \psi_{I}$ is a self-diffeomorphism of $\mathring{\Delta}^{I} \times [0, 1) \times \Delta^{J}/\sim$.
		\par\indent
		Since $\psi_I$ sends $[(1-t)(0)+tx',\tau,y]$ to $(1-\tau)\{(1-t)(0)+tx'\}+\tau y$, the composite $\psi'^{-1}\psi_I \ :\mathring{\Delta}^I\times[0,1)\times\Delta^J/\sim\longrightarrow\mathring{\Delta}^I\times[0,1)\times\Delta^J/\sim$ is given by 
		$$[(1-t)(0)+tx',\tau,y]\longmapsto\left[(1-\tau)(1-t)(0)+\{1-(1-\tau)(1-t)\}x',\frac{\tau}{1-(1-\tau)(1-t)},y\right],$$
		which is obviously smooth (see Lemma \ref{secondproperty}). Since the composite $\psi^{-1}_I\psi'  :\mathring{\Delta}^I\times[0,1)\times\Delta^J/\sim\longrightarrow\mathring{\Delta}^I\times[0,1)\times\Delta^J/\sim$ is given by
		$$[(1-s)(0)+sx',\sigma,y]\longmapsto\left[ \frac{1-s}{1-s\sigma}(0)+\frac{s-s\sigma}{1-s\sigma}x',s\sigma,y\right],$$
		$\psi^{-1}_I\psi'$ is also smooth (see Lemma \ref{secondproperty}). Hence, $\psi'^{-1}\psi_I$ is a diffeomorphism.\vspace{0.2cm} \\
		Step 3. We can easily see that $\varPhi_{I}$ is just the inverse of the diffeomorphism
		\par
		\hspace{3cm}	$\mathring{\Delta}^k \times \Delta^{p-k}_{\hat{0}} \cong \mathring{\Delta}^I \times [0, 1) \times \Delta^{J}/\sim \ \xrightarrow[\cong]{\ \  \psi_I\ \ } U_I.$
	\end{proof}
	\begin{defn}\label{gnbd}
		For $I=\{i_0,\ldots,i_k\}$ $(0\leq k<p)$ and $0 < \epsilon \leq 1$, $U_I(\epsilon)$ denotes the diffeological subspace of $\Delta^p$ defined by
		$$
		U_{I}(\epsilon) = \{(x_0, \ldots, x_p) \in \Delta^p \ | \ x_i > 0  \ \text{for} \ i \in I \ \text{and} \ x_{i_{0}} + \cdots + x_{i_{k}} > 1-\epsilon \},
		$$
		which is called a {\sl good neighborhood} of an open simplex $(i_0, \ldots, i_{k})$ of $\Delta^{p}$. 
	\end{defn}
	For $0<\epsilon\leq 1$, $\Delta^{l}_{\hat{i}}(\epsilon)$ denotes the diffeological subspace of $\Delta^l$ defined by
	$$
	\Delta^{l}_{\hat{i}}(\epsilon)=\{(x_0, \ldots, x_l) \in \ \Delta^l \ | \ x_i > 1-\epsilon \}.
	$$
	By Proposition \ref{goodnbd}, we have the diffeomorphism
	\[
	\varPhi_{I}: U_{I}(\epsilon) \xrightarrow[\ \ \cong\ \ ]{} \mathring{\Delta}^{k} \times \Delta^{p-k}_{\hat{0}}(\epsilon).
	\]
	\begin{rem}\label{gn}
		This remark, which relates to Proposition \ref{goodnbd}, is used in Section 5.\\
		$(1)$
		From Proposition \ref{goodnbd}, the composite
		$$
		U_I \xrightarrow{\varPhi_I} \mathring{\Delta}^{k} \times \Delta^{p-k}_{\hat{0}} \xrightarrow{proj} \Delta^{p-k}_{\hat{0}}
		$$
		extends to an affine surjection from $\Delta^{p}$ to $\Delta^{p-k}$.\\
		$(2)$
		From the definition of $\varPhi_{I}$, we see that the map
		$$
		\varPhi_{I} : U_{I \ \mathrm{sub}} \longrightarrow \mathring{\Delta}^{k}_{\mathrm{sub}} \times \Delta^{p-k}_{\hat{0} \ \mathrm{sub}}
		$$
		is also a diffeomorphism.
	\end{rem}
	\section{Verification of Axiom 2}
	In this section, we verify Axiom 2.
	Recall that a map $f: A \longrightarrow B$ between convex sets is affine if $f$ preserves convex combinations. We do not assume that an affine map $f : \Delta^p \longrightarrow \Delta^q $ sends vertices to vertices.
	\begin{prop}{\bf(Axiom 2)}\label{Axiom 2}
		Any affine map $f :\Delta^p\longrightarrow\Delta^q$ is smooth.
	\end{prop}
	For the proof of Proposition \ref{Axiom 2}, we need several lemmas.
	\par\indent
	Define the affine maps $s^k  :\Delta^{p+1}\longrightarrow\Delta^p$ $(0\leq k\leq p) $ by 
	$$
	s^k((i))= 
	\begin{cases}
	\  \ \ 	(i)  \ \ \ \ \,  \mathrm{\ for\ } i\leq k \\
	(i-1) \ \mathrm{\ for\ } i \geq k +1.
	\end{cases}
	$$
	First, we show that $s^k :\Delta^{p+1} \longrightarrow \Delta^p$ is smooth for $p\geq0$ and $0\leq k \leq p$. By Lemma \ref{permutation}, it is sufficient to show that $s^p : \Delta^{p+1} \longrightarrow \Delta^p$ is smooth for $p\geq 0$. For the proof, we introduce the cone construction.
	\par\indent
	The {\sl cone}, $CA$, of a  diffeological space $A$ is defined to be the quotient diffeological space $A\times I/A\times\{0\}$. The {\sl open cone}, $\mathring{C}A$, of a diffeological space $A$ is defined to be the diffeological subspace $CA-A \times \{1 \}$ of $CA$.
	
	\begin{lem}\label{cone}
		Let $A$ be a diffeological space.\\
		$(1)$ $CA$ contains the base point $\ast$ as a $\dcal$-deformation retract.\\
		$(2)$ $\mathring{C}A$ is an open diffeological subspace of $CA$ that is isomorphic to the quotient diffeological space $A \times [0, 1)/ A \times \{0 \}$.\\
		$(3)$ $CA-\ast$ is isomorphic to $A \times (0, 1]$.
		\begin{proof} (1) We construct a $\dcal$-deformation of $CA$ onto the base point.
			\par\indent
			Define the map $H :A\times I\times I\longrightarrow A\times I$ by $H(a,t,s)=(a,ts)$. Let $\pi :A\times I\longrightarrow CA$ denote the canonical $\dcal$-quotient map, and define the map $h :CA\times I \longrightarrow CA$ by the commutative diagram
			\begin{center}
				\begin{tikzcd}
					A\times I\times I \arrow{r}{H} \arrow{d}{\pi\times 1} & A\times I \arrow{d}{\pi}\\
					CA\times I \arrow{r}{h} & CA.
				\end{tikzcd}
			\end{center}
			Since $\pi\times 1 $ is final (Lemma \ref{times X}), $h$ is a smooth map, which is the desired deformation of $CA$ onto the base point.\\
			$(2)$ Consider the commutative diagram
			$$
			\begin{tikzcd}
			A\times [0, 1) \arrow[hook]{r} \arrow[swap]{d}{\pi} & A\times I \arrow{d}{\pi}\\
			\mathring{C}A \arrow[hook]{r} & CA,
			\end{tikzcd}
			$$
			and use an argument similar to that in the proof of Lemma \ref{secondproperty}.\\
			$(3)$ Consider the commutative diagram
			$$
			\begin{tikzcd}
			A\times (0, 1] \arrow[hook]{r} \arrow[swap]{d}{\pi} & A\times I \arrow{d}{\pi}\\
			CA-\ast \arrow[hook]{r} & CA.
			\end{tikzcd}
			$$
			By an argument similar to that in the proof of Lemma \ref{secondproperty}, the left vertical arrow is a $\dcal$-quotient map that is bijective, and hence, a diffeomorphism.
		\end{proof}
	\end{lem}
	We define the map
	$$
	\varphi_{p}: \Delta^{p-1} \times I \longrightarrow \Delta^p
	$$
	by $\varphi_{p}(x, t) = (1-t)(p) + td^{p}(x)$, which is an obvious extension of the map $\varphi_{p}$ in Definition \ref{simplices}. Let $\Delta^{p}_{\mathrm{cone}}$ be the set $\Delta^{p}$ endowed with the quotient diffeology by $\varphi_{p} : \Delta^{p-1} \times I \longrightarrow \Delta^{p}$. 
	Then, the $\dcal$-deformation in the proof of Lemma \ref{cone}(1) specializes to the $\dcal$-deformation
	\begin{eqnarray*}
		h:\Delta^p_{\mathrm{cone}}\times I& \longrightarrow &\Delta^{p}_{\mathrm{cone}}
	\end{eqnarray*}
	of $\Delta^p_{\mathrm{cone}}$ onto the vertex $(p)$; explicitly, $h(z, s) = (1-s)(p)+sz.$
	\par\indent
	The smoothness of $s^p$ is shown in the following lemma; note that $s^0 : \Delta^1 \longrightarrow \Delta^0$ is obviously smooth.
	\begin{lem}\label{sp}
		For all $p\geq 1$,
		\begin{itemize}
			\item[$(1)_p$]
			$id : \Delta^p\longrightarrow\Delta^p_{\mathrm{cone}}$ is smooth.
			
			\item[$(2)_p$]
			$h : \Delta^p \times [0,1) \longrightarrow\Delta^p_{\hat{p}}$ is smooth.
			
			\item[$(3)_p$]
			$s^p : \Delta^{p+1}\longrightarrow\Delta^p$ is smooth.
			
		\end{itemize}
		\begin{proof}
			We arrange the statements in the following order:
			$$(1)_1, \ (2)_1, \ (3)_1, \ldots, (1)_{p}, \ (2)_{p}, \ (3)_{p}, \ (1)_{p+1}, \ldots.$$ Since $(1)_p$ obviously holds for $p = 1$, we inductively prove all statements in three steps.
			\vspace{0.2cm}\\
			Step 1. Suppose that the implications hold up to $(1)_{p}$.
			Let us consider the commutative diagram
			\begin{equation*}
			\begin{tikzcd}
			\Delta^{p}\times [0,1) \arrow{r}{id \times 1} \arrow[swap]{d}{h} & \Delta^{p}_{\mathrm{cone}} \times [0,1) \arrow{d}{h}\\
			\Delta^p_{\hat{p}}\arrow{r}{id} & \Delta^{p}_{\mathrm{cone}} -\Delta^{p-1}_{(p)}.
			\end{tikzcd}
			\end{equation*}
			The right vertical arrow $h$ is a restriction of the smooth map $h : \Delta^{p}_{\mathrm{cone}} \times I \longrightarrow \Delta^{p}_{\mathrm{cone}},$ and the upper horizontal arrow $id\times 1 $ is smooth by $(1)_p$. Since the lower horizontal arrow is an isomorphism by Lemma \ref{halfopen} and Lemma \ref{cone}(2), $(2)_p$ holds.\vspace{0.2cm} \\
			Step 2. Suppose that the implications hold up to $(2)_p$.
			For $i\neq p, \ p+1$, we have the commutative diagram
			\begin{equation*}
			\begin{tikzcd}
			\Delta^{p}\times [0,1) \arrow{r}{s^{p-1}\times 1} \arrow[swap]{d}{\varphi_i} & \Delta^{p-1} \times [0,1) \arrow{d}{\varphi_i}\\
			\Delta^{p+1}_{\hat{i}}\arrow{r}{s^p} & \Delta^{p}_{\hat{i}}.
			\end{tikzcd}
			\end{equation*}
			Since $s^{p-1}\times 1$ is smooth by $(3)_{p-1}$ and the left vertical arrow $\varphi_i$ is final (Lemma \ref{halfopen}), $s^p$ is smooth on $\Delta^{p+1}_{\hat{i}}$ for $i \neq p,p+1$. 

			For $i = p, p+1$, we have the commutative diagram
			\begin{equation*}
			\begin{tikzcd}[column sep=6em]
			\Delta^{p}\times[0,1)  \arrow{dr}{h}  \arrow[swap]{d}{\varphi_i}  
			& \phantom{A} \\
			\Delta^{p+1}_{\hat{i}}  \arrow{r}{s^p}  &\Delta^p_{\hat{p}}.
			\end{tikzcd}
			\end{equation*}
			Since $\varphi_i$ is final (Lemma \ref{halfopen}), $s^p$ is smooth on $\Delta^{p+1}_{\hat{i}}$ by $(2)_p$. \vspace{0.2cm}\\
			Step 3. Suppose that the implications hold up to $(3)_p$. Let us show that $(1)_{p+1}$ holds. Since the map $id  :\Delta^{p+1}\longrightarrow\Delta^{p+1}_{\mathrm{cone}}$ restricts to the diffeomorphism $id : \Delta^{p+1}-\Delta^{p}_{(p+1)}\longrightarrow \Delta^{p+1}_{\mathrm{cone}}-\Delta^{p}_{(p+1)}$ (Lemma \ref{halfopen} and Lemma \ref{cone}(2)), we need to only show that $id : \Delta^{p+1}\longrightarrow \Delta^{p+1}_{\mathrm{cone}}$ is smooth near the face $\Delta^p_{(p+1)}$. The $\Delta^p$-component and $(0,1]$-component of the composite
			$$
			\Delta^{p+1}-\{(p+1)\} \xrightarrow[]{\ id\ } \Delta^{p+1}_{\mathrm{cone}}-\{(p+1)\} \xrightarrow[\cong]{\ \varphi_{p+1}^{-1}\ }\Delta^p\times (0,1]
			$$
			are denoted by $\lambda$ and $\mu$, respectively (see Lemma \ref{cone}(3)); explicitly,
			$$
			\lambda ((1-t)(p+1)+td^{p+1}(x))=x \hbox{ and } \mu((1-t)(p+1)+td^{p+1}(x))=t.
			$$
			By Lemma \ref{firstproperty}, $\mu$ is smooth. For the smoothness of $\lambda$, we show that $\lambda$ is smooth on $\Delta^{p+1}_{\hat{i}}$ for $i\neq p+1$. For simplicity, suppose that $i=p$ (Lemma \ref{permutation}). Since $\lambda|_{\Delta^{p+1}_{\hat{p}}}$ sends $(1-t)(p+1)+td^{p+1}((1-s)(p)+sd^{p}(y))\in\Delta^{p+1}_{\hat{p}}$ to $(1-s)(p)+sd^{p}(y)\in\Delta^p$, we have the factorization
			\begin{center}
				\begin{tikzcd}[column sep=6em]
					\Delta^{p+1}_{\hat{p}}  \arrow{r}{\lambda|_{\Delta^{p+1}_{\hat{p}}}}  \arrow[swap]{d}{(s^p,\mu)}  
					& \Delta^p,\\
					U  \arrow[swap]{ur}{\kappa}  &\phantom{A} 
				\end{tikzcd}
			\end{center}
			where $U$ is the open diffeological subspace of $\Delta^p_{\hat{p}}\times (0,1]$ defined by
			$$U = \{ ((1-s)(p)+sd^{p}(y),t)\in \Delta^p_{\hat{p}}\times (0,1] \ | \ s < t\}
			$$
			and $\kappa:U\longrightarrow\Delta^p$ is the map defined by
			$$\kappa((1-s)(p)+sd^{p}(y),t)=(1-\frac{s}{t})(p)+\frac{s}{t}d^{p}(y).
			$$
			We need only to show that $\kappa$ is smooth.
			\par\indent
			Define the open diffeological subspace ${V}$ of $[0,1)\times(0,1]$ by
			$${V}=\{(s,t)\in[0,1)\times(0,1] \ | \ s < t\},$$
			and define the smooth map $\delta :{V}\longrightarrow[0,1)$ by $\delta(s,t)=\frac{s}{t}$. Now, consider the commutative diagram
			$$
			\begin{tikzcd}
			\Delta^{p-1}\times V \arrow[hook]{r} \arrow[swap]{d}{\varphi_p \times 1|_{U}} & \Delta^{p-1} \times [0,1) \times (0,1] \arrow{d}{\varphi_p \times 1}\\
			U \arrow[hook]{r} & \Delta^{p}_{\hat{p}} \times (0, 1].
			\end{tikzcd}
			$$
			We can show that this diagram is a pullback diagram in $\dcal$ by an argument similar to that in the proof of Lemma \ref{secondproperty}. Since the right vertical arrow $\varphi_p \times 1$ is a $\dcal$-quotient map (Lemmas \ref{halfopen} and \ref{times X}), the left vertical arrow $\varphi_p \times 1|_{U}$ is also a $\dcal$-quotient map (Lemma \ref{initialfinal}(2)).
			Next, consider the commutative diagram
			\begin{center}
				\begin{tikzcd}[column sep=4em]
					\Delta^{p-1}\times V  \arrow{r}{1\times\delta} \arrow[swap]{d}{\varphi_p\times 1|_{U}}  &\Delta^{p-1}\times [0,1) \arrow{d}{\varphi_p}\\
					U \arrow{r}{\kappa}  & \Delta^p.
				\end{tikzcd}
			\end{center}
			Since $\varphi_{p} \circ (1 \times \delta)$ is smooth and $\varphi_{p} \times 1|_{U}$ is final, $\kappa$ is smooth.
		\end{proof}
	\end{lem}
	
	We use the  smoothness of the affine surjections $s^k$ to show the following lemma.
	\begin{lem}\label{cone ext}
		Let $g : \Delta^{p-1} \longrightarrow \Delta^{q}$ be a smooth map. For $i, j$ with $0\leq i \leq p$ and $0\leq j \leq q$, define the map $g\,\hat{}_{ij}  :\Delta^{p}_{\hat{i}}\longrightarrow \Delta^q$ by $g{\,\hat{}}_{ij}((1-t)(i)+td^i(x)) = (1-t)(j) + tg(x)$. Then $g{\,\hat{}}_{ij}$ is smooth.
		\begin{proof}[Proof.]
			We can assume that $i = 0$ and $j = 0$ (Lemma \ref{permutation}). Thus, we show that the map $g{\,\hat{}}_{00} :\Delta^{p}_{\hat{0}} \longrightarrow \Delta^q$ is smooth.
			\par\indent
			Define the map
			$g^{\Delta}_{00} : \Delta^{p}_{\hat{0}} \longrightarrow \Delta^{q+1}_{\hat{0}}$
			by the commutativity of the diagram
			\begin{center}
				\begin{tikzcd}[column sep=4em]
					\Delta^{p-1}\times [0, 1)  \arrow{r}{g \times 1} \arrow[swap]{d}{\varphi_0}  &\Delta^{q}\times [0,1) \arrow{d}{\varphi_0}\\
					\Delta^{p}_{\hat{0}} \arrow{r}{g^{\Delta}_{00}}  & \Delta^{q+1}_{\hat{0}}.
				\end{tikzcd}
			\end{center}
			Since $\varphi_{0}\circ(g \times 1)$ is smooth and the left vertical arrow $\varphi_{0}$ is final (Lemma \ref{halfopen}), $g^{\Delta}_{00}$ is smooth. We can easily see that the composite
			$$
			\Delta^{p}_{\hat{0}} \xrightarrow{\ \ \ \ g^{\Delta}_{00}\ \ \ \ } \Delta^{q+1}_{\hat{0}} \xrightarrow{\ \ \ \ \ s^0 \ \ \ \ \ } \Delta^q
			$$
			is just the map $g{\,\hat{}}_{00}$, and hence, that $g{\,\hat{}}_{00}$ is smooth by Lemma \ref{sp}.
		\end{proof}	
		
		\begin{proof}[Proof of Proposition \ref{Axiom 2}]
			
			We prove the result by induction on $p$. The result obviously holds for $p=0$. Suppose that the result holds up to $p-1$. We show that $f$ is smooth on $\Delta^{p}_{\hat{i}}$ for any $i$. There are three cases to be checked.\\
			{\sl Case 1: The image $f((i))$ is in $\mathring{\Delta}^q$}. Consider the commutative diagram
			\begin{center}
				\begin{tikzcd}
					\Delta^{p}_{\hat{i}}  \arrow{r}{f} \arrow[swap]{d}{id}  &\mathring{\Delta}^q \arrow{d}{id}\\
					\Delta^{p}_{\hat{i} \ \mathrm{sub}} \arrow{r}{f}  & \ \mathring{\Delta}^{q}_{\text{sub}},
				\end{tikzcd}
			\end{center}
			and note that the left vertical arrow is smooth (Lemma \ref{firstproperty}) and that the right vertical arrow is a diffeomorphism (Lemma \ref{secondproperty}). Since $f :\Delta^{p}_{\mathrm{\text{sub}}} \longrightarrow \Delta^{q}_{\text{sub}}$ is obviously smooth, $f  :\Delta^{p}_{\hat{i}} \longrightarrow \mathring{\Delta}^q$ is also smooth.\vspace{0.08cm}\\
			{\sl Case 2: The image $f((i))$ is a vertex of $\Delta^q$.} The composite
			$$
			\Delta^{p-1} \xhookrightarrow{\ \ \ \ d^i\ \ \ \ } \Delta^{p} \xrightarrow{\ \ \ \ f\ \ \ \ } \Delta^{q}
			$$
			is smooth by the induction hypothesis. Set $(j) = f((i))$ and consider the smooth map $(f \circ d^i)\hat{}_{ij} : \Delta^{p}_{\hat{i}} \longrightarrow \Delta^q$ (Lemma \ref{cone ext}). We can easily see that $(f \circ d^i)\hat{}_{ij}$ is just $f|_{\Delta^{p}_{\hat{i}}}$.\vspace{0.08cm}\\
			{\sl Case 3: The image $f((i))$ is in some open simplex $(i_0,\ldots,i_k)$ with $0<k<q$.} We show that the $\mathring{\Delta}^k$- and $\Delta^{q-k}_{\hat{0}}$-components of the composite $\Delta^{p}_{\hat{i}} \xrightarrow{f} U_I \xrightarrow[\cong]{\varPhi_{I}} \mathring{\Delta}^k \times \Delta^{q-k}_{\hat{0}}$ are smooth, where $I = \{i_0,\ldots,i_k \}$ (see Proposition \ref{goodnbd}).
			\par\indent
			Since the composite
			\[
			\Delta^{p}_{\hat{i} \ \mathrm{sub}} \xrightarrow{\ \ \ f \ \ \ } U_{I \ \mathrm{sub}} \xrightarrow{p_{\mathring{\Delta}^{k}}\circ \varPhi_{I}} \mathring{\Delta}^{k}_{\mathrm{sub}}
			\]
			is smooth (Remark \ref{gn}(2)), the composite
			\[
			\Delta^{p}_{\hat{i}} \xrightarrow{\ \ \ f \ \ \ } U_{I} \xrightarrow{p_{\mathring{\Delta}^{k}}\circ \varPhi_{I}} \mathring{\Delta}^{k}
			\]
			is also smooth by the argument used in Case 1.
			The $\Delta^{q-k}_{\hat{0}}$-component of $\varPhi_{I}\circ f$ is smooth by Remark \ref{gn}(1) and Case 2.
		\end{proof} 
	\end{lem}
	
	\begin{rem}\label{simplexofdelta}
		$(1)$ From Proposition \ref{Axiom 2}, we see that for $I = \{i_{0}, \ldots, i_{k} \} \subset \{0, \ldots, p \}$, $\Delta^{I}$ is just the closed simplex $\langle i_{0}, \ldots, i_{k} \rangle$ of $\Delta^{p}$ endowed with the sub-diffeology (see Section 4.2).\\
		$(2)$ By Proposition \ref{Axiom 2}, the realization functor $|\ |_{\dcal} : \scal \longrightarrow \dcal$ and the singular functor $S^{\dcal} :\dcal \longrightarrow \scal$ are defined as explained in Section 1. It is easily seen that the pair
		$$
		|\ |_{\dcal}: \scal \rightleftarrows \dcal: S^{\dcal}
		$$
		is an adjoint pair.
	\end{rem}
	\section{Key constructions for the verifications of Axioms 3 and 4}
	In this section, we construct important $\dcal$-homotopies, which are key to verifying Axioms 3 and 4.
	\par\indent
	Define the diffeological subspace $\Lambda^{p+1}_{\hat{k}}$ of $\Lambda^{p+1}_k$ by $\Lambda_{\hat{k}}^{p+1}=\Lambda^{p+1}_k-\dot{\Delta}^p_{(k)}$. 
	\begin{lem}\label{key}
		$\Lambda^{p+1}_{\hat{k}}$ is a deformation retract of $\Delta^{p+1}_{\hat{k}}$ in $\dcal$.
		\begin{proof}
			For simplicity, we assume that $k=0$ (Lemma \ref{permutation}). We construct $\mathcal{D}$-deformations $R_{p+1} :\Delta^{p+1}_{\hat{0}}\times I\longrightarrow\Delta^{p+1}_{\hat{0}}$ onto $\Lambda^{p+1}_{\hat{0}}$ by induction on $p$.
			\indent\par
			For $p=0$, we can define the $\mathcal{D}$-deformation 
			$R_1  :\Delta^{1}_{\hat{0}}\times I\longrightarrow\Delta^{1}_{\hat{0}}$
			by $$R_1((1-t)(0)+t(1),s)=(1-(1-s)t)(0)+(1-s)t(1).$$
			Suppose that $\mathcal{D}$-deformations $R_1, \ldots,R_p$ have been constructed. We construct a $\mathcal{D}$-deformation $R_{p+1}  :\Delta^{p+1}_{\hat{0}}\times I\longrightarrow\Delta^{p+1}_{\hat{0}}$ onto $\Lambda^{p+1}_{\hat{0}}$ in three steps.\vspace{0.2cm}\\
			{\sl Step 1: Construction of $R_{q,r}$ $(q\geq0,r\leq p)$}. Let $r$ $\leq$ $p$. For $q>0$, define the $\dcal$-homotopy $$R_{q,r}  :\mathring{\Delta}^q\times\Delta^{r}_{\hat{0}}\times I \longrightarrow\mathring{\Delta}^q\times\Delta^{r}_{\hat{0}}$$ by $R_{q,r}(x,y,s)=(x,R_{r}(y,\phi_q(x)s))$, where $\phi_q  :\Delta^q\longrightarrow[0,1]$ is a smooth function such that $\phi_q\equiv1$ near the barycenter $b_q$ of $\Delta^{q}$, and $\phi_q\equiv0$ near $\dot{\Delta}^q$ (Lemma \ref{firstproperty}). Define the $\dcal$-homotopy 
			$R_{0, r} : \Delta^{r}_{\hat{0}} \times I \longrightarrow \Delta^{r}_{\hat{0}}$ by $R_{0, r} = R_{r}$.\vspace{0.2cm}\\
			{\sl Step 2: Construction of $R_{A^p}$.} Set $A^p=\Delta^p-B^p$, where $B^p$ is a closed $p$-dimensional disk contained in $\mathring{\Delta}^p$ whose center is the barycenter $b_p$ of $\Delta^p$. Then, we deform $A^p$ onto $\dot{\Delta}^p$ as follows; see Definition \ref{gnbd} for a good neighborhood of an open simplex $(i_0, \ldots, i_k)$ of $\Delta^p$.
			\begin{itemize}
				\item[$\bullet$]
				We deform $A^p$ by applying $R_{p-1,1}$ to good neighborhoods of the open $(p-1)$-simplices of $\Delta^p$; the resulting diffeological subspace is denoted by $\dot{\Delta}^p\cup D^{(1)}$, where $D^{(1)}$ is a neighborhood of the $(p-2)$-skeleton $\mathrm{sk}_{p-2}\ \Delta^p$ which is $\czero$-homotopic to $\mathrm{sk}_{p-2}\ \Delta^p$.
				
				\item[$\bullet$]
				We deform $\dot{\Delta}^p\cup D^{(1)}$ by applying $R_{p-2,2}$ to good neighborhoods of the open $(p-2)$-simplices of $\Delta^p$; the resulting diffeological subspace is denoted by $\dot{\Delta}^p\cup D^{(2)}$, where $D^{(2)}$ is a neighborhood of the $(p-3)$-skeleton sk$_{p-3} \ \Delta^p$ which is $\czero$-homotopic to $\mathrm{sk}_{p-3}\ \Delta^p$.
				
				\item[$\bullet$]
				We iterate the same procedure using $R_{p-k,k}$ $(k=3,\ldots,p)$ to deform $\dot{\Delta}^p\cup D^{(2)}$ onto $\dot{\Delta}^p$.
			\end{itemize}
			\begin{center}
				\begin{tikzpicture}
				\draw[fill=gray!30] ({2*cos(-30)},{2*sin(-30)}) -- ({2*cos(90)},{2*sin(90)}) --  ({2*cos(-150)},{2*sin(-150)}) -- ({2*cos(-30)},{2*sin(-30)});
				\node at (1.5,0) [right] {$\longrightarrow$};
				\draw [fill=white, dashed] (0,0) circle [radius=0.75];
				\end{tikzpicture}
				\begin{tikzpicture}
				\draw[fill=gray!30] ({2*cos(-30)},{2*sin(-30)}) -- ({2*cos(90)},{2*sin(90)}) --  ({2*cos(-150)},{2*sin(-150)}) -- ({2*cos(-30)},{2*sin(-30)});
				\node at (1.5,0) [right] {$\longrightarrow$};
				\draw [fill=white, dashed] (0,0) circle [radius=0.75];
				
				\draw[fill=white, white] ({cos(60)},{2*sin(-30)}) -- ({cos(60)},{2*sin(-15)}) -- ({cos(120)},{2*sin(-15)}) -- ({cos(120)},{2*sin(-30)}) -- ({cos(60)},{2*sin(-30)});
				\draw[dashed] ({cos(60)},{2*sin(-30)}) -- ({cos(60)},{2*sin(-15)});
				\draw[dashed] ({cos(120)},{2*sin(-30)}) -- ({cos(120)},{2*sin(-15)});
				
				\draw[rotate=240,fill=white,white]  ({cos(60)},{2*sin(-30)}) -- ({cos(60)},{2*sin(-15)}) -- ({cos(120)},{2*sin(-15)}) -- ({cos(120)},{2*sin(-30)}) -- ({cos(60)},{2*sin(-30)});
				\draw[dashed,rotate=240] ({cos(60)},{2*sin(-30)}) -- ({cos(60)},{2*sin(-15)});
				\draw[dashed,rotate=240] ({cos(120)},{2*sin(-30)}) -- ({cos(120)},{2*sin(-15)});

				\draw[rotate=120,fill=white,white]  ({cos(60)},{2*sin(-30)}) -- ({cos(60)},{2*sin(-15)}) -- ({cos(120)},{2*sin(-15)}) -- ({cos(120)},{2*sin(-30)}) -- ({cos(60)},{2*sin(-30)});
				\draw[dashed,rotate=120] ({cos(60)},{2*sin(-30)}) -- ({cos(60)},{2*sin(-15)});
				\draw[dashed,rotate=120] ({cos(120)},{2*sin(-30)}) -- ({cos(120)},{2*sin(-15)});
				\draw ({2*cos(-30)},{2*sin(-30)}) -- ({2*cos(90)},{2*sin(90)}) --  ({2*cos(-150)},{2*sin(-150)}) -- ({2*cos(-30)},{2*sin(-30)});
				\end{tikzpicture}
				\begin{tikzpicture}
				\draw ({2*cos(-30)},{2*sin(-30)}) -- ({2*cos(90)},{2*sin(90)}) --  ({2*cos(-150)},{2*sin(-150)}) -- ({2*cos(-30)},{2*sin(-30)});
				\node at (1.5,0) [right] {};
				\end{tikzpicture}
				\\
				{\bf Fig. 6.1} $\dcal$-deformation of $A^{2}$ onto $\dot{\Delta}^2$
			\end{center}
			The $\mathcal{D}$-deformation of $A^p$ onto $\dot{\Delta}^p$ described above is denoted by $R_{A^p}$.\vspace{0.2cm}\\
			{\sl Step 3: Construction of $R_{p+1}$.} In the construction of $R_{A^p}$, we take a closed $p$-dimensional disk $B^p$ in $\mathring{\Delta}^p$ such that $B^p$ is contained in the interior $\phi^{-1}_p(1)^{\circ}$ of $\phi^{-1}_p(1)$.
			Define the $\mathcal{D}$-deformation $R_{A^p\times [0,1)}  :A^p\times[0,1)\times I\longrightarrow A^p\times [0,1)$ of $A^p\times[0,1)$ onto $\dot{\Delta}^p\times[0,1)$ to be the composite
			$$A^p\times[0,1)\times I \cong A^p\times I \times [0,1)\xrightarrow{R_{A^p}\times 1}A^p\times[0,1).$$
			Let us construct the desired $\mathcal{D}$-deformation $R_{p+1}$ using $R_{p,1}$ and $R_{A^p\times[0,1)}$.
			\par				
			First, under the canonical identification of $[0, 1)$ and $\Delta^1_{\hat{0}}$, we define the $\dcal$-homotopy
			$$
			P'_{p+1}: \Delta^p \times [0, 1) \times I \longrightarrow \Delta^p \times [0, 1)
			$$
			to be the obvious extension of $R_{p,1} : \mathring{\Delta}^p \times [0, 1) \times I \longrightarrow \mathring{\Delta}^p \times [0, 1);$ explicitly,
			$$
			P_{p+1}'(x,t,s)=
			\begin{cases}
			R_{p,1}(x,t,s) & \text{if}\, \ x \in \mathring{\Delta}^p\\
			(x, t) & \text{if} \, \ x \in \dot{\Delta}^p.
			\end{cases}
			$$
			Then, we define the map
			$$
			P_{p+1}: \Delta^{p+1}_{\hat{0}} \times I \longrightarrow \Delta^{p+1}_{\hat{0}}
			$$
			by the commutativity of the diagram
			\begin{center}
				\begin{tikzcd}
					\Delta^{p} \times [0, 1) \times I  \arrow{r}{P'_{p+1}} \arrow[swap]{d}{\varphi_0 \times 1}  &\Delta^p \times [0, 1) \arrow{d}{\varphi_0}\\
					\Delta^{p+1}_{\hat{0}}\times I \arrow{r}{P_{p+1}}  & \Delta^{p+1}_{\hat{0}}.
				\end{tikzcd}
			\end{center}
			Since $\varphi_0 \times 1$ is final (Lemmas \ref{halfopen} and \ref{times X}), $P_{p+1}$ is a $\dcal$-homotopy.
			\par\indent
			Second, we define the $\dcal$-homotopies
			$$
			Q'_{p+1}: \phi^{-1}_{p}(1)^{\circ} \times [0, 1) \times I \longrightarrow \Delta^{p} \times [0, 1), 
			$$
			$$
			Q''_{p+1}: A^{p} \times [0, 1) \times I \longrightarrow \Delta^{p} \times [0, 1)
			$$
			by
			$
			Q'_{p+1}(x, t, s) = (x, 0),
			$
			$
			Q''_{p+1}(x, t, s) = R_{A^{p}\times [0, 1)}(P'_{p+1}(x, t, 1), s)
			$ respectively. Consider the composites of $Q'_{p+1}$ and $Q''_{p+1}$ with $\,\varphi_{0}: \Delta^{p} \times [0, 1) \longrightarrow \Delta^{p+1}_{\hat{0}}$, and observe that
			\[
			\varphi_{0} \circ Q'_{p+1} = \varphi_{0} \circ Q''_{p+1} = \text{the constant map onto the vertex (0)}
			\] 
			on $\phi^{-1}_{p}(1)^{\circ} \times [0, 1) \times I \cap A^{p} \times [0, 1) \times I$. Then, we have the $\dcal$-homotopy
			$$
			\varphi_{0}\circ Q'_{p+1} + \varphi_{0} \circ Q''_{p+1}: \Delta^p \times [0, 1) \times I \longrightarrow \Delta^{p+1}_{\hat{0}}.
			$$
			Define the map $Q_{p+1} : \Delta^{p+1}_{\hat{0}} \times I \longrightarrow \Delta^{p+1}_{\hat{0}}$ by the commutativity of the diagram
			
			\begin{center}
				\begin{tikzcd}
					\Delta^{p} \times [0, 1) \times I  \arrow{dr}{\varphi_{0} \circ Q'_{p+1} + \varphi_{0} \circ Q''_{p+1}} \arrow[swap]{d}{\varphi_0 \times 1} \\
					\Delta^{p+1}_{\hat{0}}\times I \arrow{r}{Q_{p+1}}  & \Delta^{p+1}_{\hat{0}}.
				\end{tikzcd}
			\end{center}
			Since $\varphi_{0} \times 1$ is final (Lemmas \ref{halfopen} and \ref{times X}), $Q_{p+1}$ is a $\dcal$-homotopy connecting $P_{p+1}(\cdot, 1)$ to a retraction of $\Delta^{p+1}_{\hat{0}}$ onto $\Lambda^{p+1}_{\hat{0}}$.
			\par\indent
			Last, define the $\dcal$-homotopy
			$$
			R_{p+1}: \Delta^{p+1}_{\hat{0}} \times I \longrightarrow \Delta^{p+1}_{\hat{0}}
			$$
			to be the composite of the $\dcal$-homotopies $P_{p+1}$ and $Q_{p+1}$ (see Remark \ref{Dhomotopy}). Then, $R_{p+1}$ is the desired $\dcal$-deformation.
		\end{proof}
	\end{lem}
	
	From the proof of Lemma \ref{key}, we have the following result.
	
	\begin{prop}\label{deformontbdry}
		$\dot{\Delta}^p$ is a deformation retract of $\Delta^p - \{b_p\}$ in $\dcal$, where $b_p$ is the barycenter of $\Delta^p$.
		\begin{proof}
			Let $B^p$ be a small $p$-dimensional disk in $\Delta^p$ centered at $b_p$. Since $\Delta^p - \{b_p\}$ can be deformed into $\Delta^p - B^p$, $\Delta^p - \{b_p\}$ is deformed onto $\dot{\Delta}^p$ via the $\mathcal{D}$-deformation $R_{A^p}$ constructed in the proof of Lemma \ref{key}.
		\end{proof}
	\end{prop}
	We end this section with a simple lemma on $\dcal$-retracts (i.e., retracts in the category $\dcal$) and $\dcal$-deformation retracts, which is used in the succeeding sections.
	
	\begin{lem}\label{predefretr}
		Let $A$ be a diffeological subspace of $B$, and let $i$ be the inclusion of $A$ into $B$. Consider the pushout diagram in $\dcal$
		\begin{center}
			\begin{tikzcd}
				A  \arrow[hook]{r}{i} \arrow[swap]{d}{f}  & B \arrow{d}{g}\\
				C \arrow[hook,swap]{r}{j}  & B \underset{A}{\cup} C.
			\end{tikzcd}
		\end{center}
		$(1)$ If $A$ is a $\dcal$-retract of $B$, then $C$ is also a $\dcal$-retract of $B \underset{A}{\cup} C$.\\
		$(2)$ If $A$ is a $\dcal$-deformation retract of $B$, then $C$ is also a $\dcal$-deformation retract of $B \underset{A}{\cup} C$.
	\end{lem}
	\begin{proof}
		$(1)$ Obvious.\\
		$(2)$ Let $R: B \times I \longrightarrow B$ be a $\dcal$-deformation of $B$ onto $A$. Since the diagram
		$$
		\begin{tikzcd}
		A \times I \arrow[hook]{r}{i \times 1} \arrow[swap]{d}{f \times 1}  
		&  B \times I  \arrow{d}{g \times 1}\\
		C \times I  \arrow[hook,swap]{r}{j \times 1}  
		&  (B\underset{A}{\cup}C) \times I
		\end{tikzcd}
		$$
		is a pushout diagram in $\dcal$ (Proposition \ref{category D}(2)), we have the smooth map
		$$
		(B\underset{A}{\cup}C) \times I = B \times I \underset{A \times I}{\cup} C \times I \xrightarrow{g \circ R + j \circ proj} B\underset{A}{\cup}C,
		$$
		which is the desired $\dcal$-deformation.			
	\end{proof}	
	\section{Verification of Axiom 3}
	In this section, we verify Axiom 3.
	\begin{prop}[\bf Axiom 3]\label{Axiom 3}
		The canonical smooth injection
		$$\left| \dot{\Delta}[p] \right|_{\dcal} \longhookrightarrow \Delta^p$$
		is a $\dcal$-embedding.
	\end{prop}	
	Let us begin by showing the following three lemmas.	
	\begin{lem}\label{subspotcolim}
		Let $J$ be a small category and 
		$Y_{\bullet}  : J \longrightarrow \dcal$
		be a functor. Set 
		$Y = \underset{j\in J}{\mathrm{colim}} \ Y_j$, 
		and let $B$ be a diffeological subspace of $Y$. 
		Then, the equality $B = \underset{j\in J}{\mathrm{colim}} \ B_j$ holds in the category $\dcal$, where $B_j$ is the inverse image of $B$ by the canonical map 
		$\phi_j : Y_j \longrightarrow Y$.
		\begin{proof}
			We can easily see that the diffeological spaces $B$ and $\underset{j\in J}{\mathrm{colim}} \ B_{j}$ coincide as sets (see the proof of Proposition \ref{category D}(1)). Thus, consider the pullback diagram in $\dcal$:
			\begin{center}
				\begin{tikzcd}
					\underset{j}{\coprod} \ B_j  \arrow[hook]{r} \arrow[swap]{d}{\scriptsize{\sum}\ \phi_j}  &\underset{j}{\coprod}\ Y_j \arrow{d}{\sum\phi_j}\\
					B \arrow[hook]{r}  & Y
				\end{tikzcd}
			\end{center}
			(see Lemma \ref{initial}). Since the right vertical arrow is final, Lemma \ref{initialfinal}(2) implies that the left vertical arrow is also final, and hence, $B = \underset{j\in J}{\mathrm{colim}}\,\,B_j$ holds in $\dcal$.		
		\end{proof}
	\end{lem}
	Let $\Delta$ denote the category of finite ordinal numbers. Recall that $\scal = Set^{\Delta^{op}}$ and that the Yoneda embedding $\Delta \longhookrightarrow \scal$ assigns to $[p]$ the standard $p$-simplex $\Delta[p]$. For a subcomplex $K$ of $\Delta[p]$, define the category $(\Delta\downarrow K)_{inj}$ to be the full subcategory of $\Delta\downarrow K$ consisting of objects 
	$\sigma: \Delta[n] \longrightarrow K$ whose composite with $K \longhookrightarrow \Delta[p]$ correspond to injective monotonic maps via the Yoneda embedding. Define the subcategory $I_{K}$ of $\dcal$ to be 
	the image of the functor $(\Delta\downarrow K)_{inj} \longrightarrow \dcal$ which assigns to $\sigma: \Delta[n] \longrightarrow K$ to $\Delta^{n}$.
	\begin{lem}\label{final}
		For a subcomplex $K$ of $\Delta[p]$, the realization $|K|_{\dcal}$ is canonically isomorphic to $\underset{I_{K}}{\mathrm{colim}}\ \Delta^{n}$.
	\end{lem}
	\begin{proof}
		Since $(\Delta\downarrow K)_{inj}$ is final in $\Delta\downarrow K$, $|K|_{\dcal}$ is canonically isomorphic to $\underset{{(\Delta \downarrow K)_{inj}}}{\mathrm{colim}}\Delta^{n}$ (\cite[p. 217]{Mac}), and hence to $\underset{I_{K}}{\mathrm{colim}}\ \Delta^{n}$.
	\end{proof}
	The {\sl cone} $CL$ of a simplicial set $L$ is defined by
	\[
	CL = \underset{\Delta\downarrow L}{\mathrm{colim}} \ \Delta[m+1]
	\]
	in the category $\scal_{\ast}$ of pointed simplicial sets, where the base point of $\Delta[m+1]$ is the vertex (0) (see \cite[p. 189]{GJ} for the precise definition of the functor $\Delta\downarrow L \longrightarrow \scal_{\ast}$ sending $\sigma: \Delta[m] \longrightarrow X$ to $\Delta[m+1]$). In order to work in the unpointed context, let us express $CL$ as a colimit in the category $\scal$. Define the {\sl enlarged simplex category} $(\Delta\downarrow L)^{+}$ to be the full subcategory of the overcategory $\scal\downarrow L$ consisting of objects $\sigma: \Delta[m] \longrightarrow L \ (m \geq -1)$, where $\Delta[-1]$ denotes the initial object of $\scal$. Then, the functor $\Delta\downarrow L \longrightarrow \scal_{\ast}$ extends a functor $(\Delta\downarrow L)^{+} \longrightarrow \scal_{\ast}$ in an obvious manner. For the composite of this functor with the forgetful functor $\scal_{\ast} \longrightarrow \scal$, we have the following result.
	\begin{lem}\label{simplicialcone}
		For a simplicial set $L$, the cone $CL$ is naturally isomorphic to $\underset{({\Delta\downarrow L})^{+}}{\mathrm{colim}} \ \Delta[m+1]$, where the colimit is formed in the category $\scal$.
	\end{lem}
	\begin{proof}
		Obvious.
	\end{proof}
	\begin{proof}[Proof of Proposition \ref{Axiom 3}]
		In the proof, we use the following notation: for a subset $A$ of a diffeological space $X$, $A_X$ is the set $A$ endowed with the sub-diffeology of $X$.
		\par\indent
		For a subcomplex $L$ of $\Delta[p]$, let $|L|$ be the underlying set of $|L|_\dcal$, and regard it as a subset of $\Delta^p$ via the canonical injection. Setting $K = \dot{\Delta}[p]$, we show that the smooth map 
		$id :|K|_\dcal \longrightarrow |K|_{\Delta^p}$ 
		is a diffeomorphism.
		\par\indent
		For $p \leq 1$, the result obviously holds. Suppose that the result holds up to $p-1$. Then, we show that the result holds for $p$.
		\par\indent
		Set 
		$|K|_{\hat{i}} = |K| \, \cap \, \Delta^{p}_{\hat{i}}$,
		and let $|K|_{\hat{i} \ \dcal}$ be the set $|K|_{\hat{i}}$ endowed with the sub-diffeology of $|K|_{\dcal}$. We need to only show that 
		$id  : |K|_{\hat{i} \ \dcal} \longrightarrow |K|_{\hat{i} \ \Delta^p}$
		is a diffeomorphism for any $i$.
		\par\indent
		For $i \neq 0$, let $\gamma$ denote the permutation of $\{(0), \ldots, (p) \}$, switching $(0)$ and $(i)$. The affine bijection of $\Delta^{p}$ induced by $\gamma$ is also denoted by $\gamma$. Then, 
		$
		\gamma|_{|K|}: |K|_{\Delta^p} \longrightarrow |K|_{\Delta^p}
		$
		is a diffeomorphism (Lemma \ref{permutation}). We can see that
		$
		\gamma|_{|K|}: |K|_{\dcal} \longrightarrow |K|_{\dcal}
		$
		is also a diffeomorphism using Lemmas \ref{final} and \ref{permutation}. (Note that $\gamma$ does not define a simplicial self-map of $\dot{\Delta}[p]$ for $p > 1$.)
		Thus, we may assume that $i = 0$.
		\par\indent
		Set $K' = \Lambda_{0}[p]$ and $K'' = \Delta[p-1]$, and identify $K''$ with the $0^{th}$ face of $\Delta[p]$. Then, $\overline{K}:= \dot{\Delta}[p-1]$ is the intersection of $K'$ and $K''$.
		\par\indent
		First, we investigate the diffeology of $|K|_{\hat{0}\ \dcal}$. Since $K$ is isomorphic to the pushout $K' \underset{\overline{K}}{\cup}  K''$ in the category $\scal$, we have the isomorphism in $\dcal$ 
		\[
		|K|_{\dcal} \cong |K'|_{\dcal} \underset{|\overline{K}|_{\dcal}}{\cup} |K''|_{\dcal}
		\]
		(Remark \ref{simplexofdelta}(2)), and hence, the equality
		\[
		|K|_{\hat{0} \ \dcal} = |K'|_{\hat{0} \ \dcal}
		\]
		(Lemma \ref{subspotcolim}). Since $K'$ is isomorphic to the cone $C\overline{K}$ in a canonical way, we have
		\begin{eqnarray*}
			|K|_{\hat{0} \, \dcal} & = & |K'|_{\hat{0} \, \dcal}\\
			& = &  \underset{(\Delta\downarrow \overline{K})^{+}}{\mathrm{colim}} \ \Delta^{m+1}_{\hat{0}}
		\end{eqnarray*}
		by Lemmas \ref{subspotcolim} and \ref{simplicialcone}, and Remark \ref{simplexofdelta}(2).
		\par\indent
		Next, we look closely at the diffeology of $|K|_{\hat{0}\,\Delta^p}$. Consider the pullback diagram in $\dcal$
		\begin{center}
			\begin{tikzcd}
				(|\overline{K}| \times [0, 1) \cup \Delta^{p-1} \times (0))_{\Delta^{p-1} \times [0, 1)}  \arrow[hook]{r} \arrow[swap]{d}{\varphi_0}  &\Delta^{p-1} \times [0, 1) \arrow{d}{\varphi_0}\\
				{|K|}_{\hat{0} \, {\Delta^p}} \arrow[hook]{r}  & \Delta^{p}_{\hat{0}}
			\end{tikzcd}
		\end{center}
		(see Lemma \ref{initial}). Since $\varphi_{0}  : \Delta^{p-1} \times [0, 1) \longrightarrow \Delta^{p}_{\hat{0}}$ is final (Lemma \ref{halfopen}),
		$$
		\varphi_0 : (|\overline{K}|\times [0, 1) \cup \Delta^{p-1} \times (0))_{\Delta^{p-1} \times [0, 1)} \longrightarrow |K|_{\hat{0}\,\Delta^p}
		$$
		is also final (Lemma \ref{initialfinal}(2)). Set $V=\Delta^{p-1}-\{b_{p-1} \}$. By the definition of a final structure (see \cite[Section 2]{CSW}),
		$$
		\varphi_{0}:(|\overline{K}| \times [0, 1) \cup V \times (0))_{\Delta^{p-1}\times[0, 1)} \longrightarrow |K|_{\hat{0}\,\Delta^p}
		$$
		is also final. Since $|\overline{K}|_{\Delta^{p-1}}$ is a $\dcal$-retract of $V$ (Proposition \ref{deformontbdry}), there is a retract diagram in $\dcal$ of the form
		\[
		{|\overline{K}|}_{\Delta^{p-1}} \times [0, 1) \xhookrightarrow[]{\ \iota \times 1\ } V \times [0, 1) \xrightarrow{\ \rho \times 1\ } |\overline{K}|_{\Delta^{p-1}} \times [0, 1).
		\]
		This diagram restricts to the retract diagram in $\dcal$
		\[
		{|\overline{K}|}_{\Delta^{p-1}} \times [0, 1) \xhookrightarrow[]{\ \iota \times 1\ } ({|\overline{K}|} \times [0, 1) \ \cup\ V \times (0))_{\Delta^{p-1}\times[0, 1)} \xrightarrow{\ \rho \times 1\ } {|\overline{K}|}_{\Delta^{p-1}} \times [0, 1),
		\]
		which shows that 
		$$
		\varphi_0 : |\overline{K}|_{\Delta^{p-1}} \times [0, 1) \longrightarrow |K|_{\hat{0}\,\Delta^p}
		$$
		is final. Since 
		$
		|\overline{K}|_{\Delta^{p-1}} \cong \underset{{(\Delta\downarrow\overline{K})}^{+}}{\text{colim}} \ \Delta^{m}
		$
		by the induction hypothesis, we have
		\begin{eqnarray*}
			|K|_{\hat{0} \ \Delta^p}& \cong &|\overline{K}|_{\Delta^{p-1}} \times [0, 1) / |\overline{K}|_{\Delta^{p-1}} \times (0) \\
			& \cong & \underset{{(\Delta\downarrow\overline{K})}^{+}}{\text{colim}} \ (\Delta^{m} \times [0, 1)) / \underset{{(\Delta\downarrow\overline{K})}^{+}}{\text{colim}} \  (\Delta^m \times (0)) \\
			& \cong & \underset{{(\Delta\downarrow\overline{K})}^{+}}{\text{colim}} \ \Delta^{m+1}_{\hat{0}}
		\end{eqnarray*}
		(see Proposition \ref{category D}(2) and Lemma \ref{halfopen}), which completes the proof.
	\end{proof}
	
	\section{Verification of Axiom 4}
	In this section, we verify Axiom 4.
	\begin{prop}[\bf Axiom 4]\label{Axiom 4}
		$\Lambda^{p+1}_k$ is a deformation retract of $\Delta^{p+1}$ in $\dcal$.
	\end{prop}
	Let us begin by showing a lemma on the diffeology of $\Delta^{p+1}$. In this section, let $\Delta^{p+1}_{\mathrm{cone}}$ denote the set $\Delta^{p+1}$ endowed with the quotient diffeology by the surjection
	$$
	\varphi_{0}: \Delta^{p} \times I \longrightarrow \Delta^{p+1}
	$$
	defined by $\varphi_{0}(x, t) = (1-t)(0) + td^{0}(x)$, which is the obvious extension of the map $\varphi_{0}$ in Definition \ref{simplices}. For a subset $A$ of $\Delta^{p+1}$, $A_{\mathrm{cone}}$ is the set $A$ endowed with the sub-diffeology of $\Delta^{p+1}_{\mathrm{cone}}$. Recall that $id  : \Delta^{p+1} \longrightarrow \Delta^{p+1}_{\mathrm{cone}}$ is smooth (Lemma \ref{sp}).
	\begin{lem}\label{deltacone}
		The equality $\Delta^{p+1}-\dot{\Delta}^{p}_{(0)}$ $=$ $(\Delta^{p+1}-\dot{\Delta}^{p}_{(0)})_{\mathrm{cone}}$ holds.
		\begin{proof}
			The map
			\[
			id: \Delta^{p+1}_{\hat{0}} \longrightarrow \Delta^{p+1}_{\hat{0} \ \mathrm{cone}}
			\]
			is a diffeomorphism by Lemmas \ref{halfopen} and \ref{cone}(2). The map
			\[
			id: U_{\{1, \ldots, p+1\}} \longrightarrow U_{\{1, \ldots, p+1 \} \ \mathrm{cone}}
			\]
			is also a diffeomorphism by Proposition \ref{goodnbd}, and Lemma \ref{cone}(3). Thus, the result holds.
		\end{proof}
	\end{lem}

	\begin{proof}[Proof of Proposition 8.1]
		For simplicity, we assume that $k=0$ (Lemma \ref{permutation}).
		\par\indent
		Let $P_{p+1}  : \Delta^{p+1}_{\hat{0}} \times I \longrightarrow \Delta^{p+1}_{\hat{0}}$ and $Q_{p+1} : \Delta^{p+1}_{\hat{0}} \times I \longrightarrow \Delta^{p+1}_{\hat{0}}$ be $\dcal$-homotopies constructed in the proof of Lemma \ref{key}.
		Choose sufficiently small $\epsilon>0$ and a monotonically increasing smooth function $\mu :[0,1]\longrightarrow[0,1]$ such that $\mu\equiv0$ on $[0,\epsilon]$ and $\mu\equiv1$ on $[2\epsilon,1]$. 
		Define the $\dcal$-homotopy $\qcal_{p+1} : \Delta^{p+1}_{\hat{0}} \times I \longrightarrow \Delta^{p+1}_{\hat{0}}$ by 
		$$
		\qcal_{p+1}((1-t)(0)+td^{0}(x), s) = Q_{p+1}((1-t)(0)+td^{0}(x), \mu(1-t)s),
		$$
		and let $\rcal_{p+1}$ be the composite of the $\dcal$-homotopies $P_{p+1}$ and $\qcal_{p+1}$ (see Remark \ref{Dhomotopy}). From the construction $\rcal_{p+1}$ extends uniquely to a smooth map from $\Delta^{p+1}_{\mathrm{cone}} \times I$ to $\Delta^{p+1}_{\mathrm{cone}}$, which is the constant homotopy of the identify near $\dot{\Delta}^{p}_{(0)}$. Thus, this extension of $\rcal_{p+1}$ can be regarded as a smooth map from $\Delta^{p+1} \times I$ to $\Delta^{p+1}$ (Lemma \ref{deltacone}), which is also denoted by $\rcal_{p+1}$.

		\par\indent 
		Using the $\dcal$-homotopies $\rcal_{p+1}$, $R_{q, r}$, and $R_{A^{p}}$ (see the proof of Lemma \ref{key}), we deform $\Delta^{p+1}$ onto $\Lambda^{p+1}_{0}$ as follows.
		
		\begin{itemize}
			\item[$\bullet$]
			We deform $\Delta^{p+1}$ by applying ${\mathcal R}_{p+1}$; the resulting diffeological subspace is denoted by $\Lambda^{p+1}_0\cup E^{(1)}$, where $E^{(1)}$ is a neighborhood of $\dot{\Delta}^p_{(0)}$ that is $\czero$-homotopic to $\dot{\Delta}^p_{(0)}$. We may assume that $A'^{p}:= E^{(1)} \cap \Delta^{p}_{(0)}$ is a diffeological subspace of the form $\Delta^{p}_{(0)}-B'^{p}$, where $B'^{p}$ is an open $p$-dimensional disk such that its closure is contained in $\mathring{\Delta}^{p}_{(0)}$ and its center is the barycenter of $\Delta^{p}_{(0)}$.
			
			\item[$\bullet$]
			We deform $\Lambda^{p+1}_0\cup E^{(1)}$ by applying $R_{p-1,2}$ to good neighborhoods in $\Delta^{p+1}$ of the open $(p-1)$-simplices of $\Delta^p_{(0)}$; the resulting diffeological subspace is denoted by $\Lambda^{p+1}_0\cup E^{(2)}\cup A'^p$, where $E^{(2)}$ is a neighborhood in $\Delta^{p+1}$ of the $(p-2)$-skeleton $\mathrm{sk}_{p-2}\ \Delta^p_{(0)}$ which is $\czero$-homotopic to $\mathrm{sk}_{p-2}\ \Delta^p_{(0)}$.
			
			\item[$\bullet$]
			We iterate the same procedure using $R_{p-k,k+1}$ $(k=2, \ldots,\ p)$ to deform $\Lambda^{p+1}_0\cup E^{(2)}\cup A'^p$ onto $\Lambda^{p+1}_0\cup A'^p$.
			
			\item[$\bullet$]
			Since the diffeological subspace 
			$\Lambda^{p+1}_{0} \cup A'^{p}$
			is just the pushout $|\Lambda_{0}[p+1]|_{\dcal} \underset{\dot{\Delta}^{p}}{\cup} A'^{p}$
			by Proposition \ref{Axiom 3}, Remark \ref{simplexofdelta}(2), and Lemma \ref{subspotcolim} (cf.~the argument in the proof of Proposition \ref{Axiom 3}), we deform 
			$\Lambda^{p+1}_{0} \cup A'^p$
			onto the diffeological subspace $|\Lambda_{0}[p+1]|_{\dcal}$ $(= \Lambda^{p+1}_{0})$ using the deformation $R_{A^p}$ constructed in the proof of Lemma \ref{key} (see Lemma \ref{predefretr}(2)). \hfill \qed			
		\end{itemize}		\color{white}
		\end{proof}
	\if\else
	\begin{lem}
		The standard $p$-simplex $\dot{\Delta}^p$ is contractible.
		\begin{proof}
			The standard $p$-simplex $\dot{\Delta}^p$ is deformed onto the vertex $(0)$ by iterated application of the deformation of a simplex onto a horn.
		\end{proof}
	\end{lem}
	\fi%
	\section{Proofs of Theorems \ref{model} and \ref{homotopygp}}
	In this section, we establish the basic properties of the singular functor $S^{\dcal}$ and the realization functor $|\ |_{\dcal}$, and prove Theorems \ref{model} and \ref{homotopygp}. As mentioned in Section 1, the proof of Theorem \ref{model} is constructed using only properties (1)-(3) of the category $\dcal$ and Axioms 1-4 for the standard simplices.
		\def\colim{{\mathrm{colim}}}
	\subsection{Singular and realization functors}
	Recall the definitions of the singular functor $S^{\dcal} : \dcal \longrightarrow \scal$ and the realization functor $|\ |_{\dcal} : \scal \longrightarrow \dcal$ (Section 1). The following proposition has already been stated in Remark \ref{simplexofdelta}(2).
	
	\begin{prop}\label{adjointSD}
$|\ |_{\dcal}: \scal \rightleftarrows \dcal: S^{\dcal}$ is an adjoint pair.
	\end{prop}
	\begin{proof} We can easily see that the natural isomorphism $\scal(K,S^{\mathcal{D}} X)\cong \mathcal{D}(|K|_{\dcal},X)$ holds for $K\in\scal$ and $X\in\dcal$.
	\end{proof}
		Define the set $\ical$ of morphisms of $\dcal$ by
		\[
			\ical =  \{\dot{\Delta}^{p} \longhookrightarrow \Delta^{p} \ | \ p\geq 0 \}.
		\]
		We generalize Proposition \ref{Axiom 3} as follows.
		\begin{prop}\label{relcellcpx}
			If $K$ is a subcomplex of a simplicial set $L$, then the canonical smooth injection
			\[
			|K|_{\dcal} \longhookrightarrow |L|_{\dcal}
			\]
			is a $\dcal$-embedding.
		\end{prop}
		\begin{proof}
			The inclusion $K \longhookrightarrow L$ is a relative $\ical_{\scal}$-cell complex, where $\ical_{\scal} = \{\dot{\Delta}[p] \longhookrightarrow \Delta[p]\ |\ p \geq 0 \}$. Thus, $|K|_{\dcal} \longhookrightarrow |L|_{\dcal}$ is a relative $\ical$-cell complex (Propositions \ref{Axiom 3} and \ref{adjointSD}), and hence a $\dcal$-embedding (see the comment after Proposition \ref{category D}).
		\end{proof}
		\begin{rem}\label{simplexcontractible}
					The horn $\Lambda^{p}_{k}$ is deformed onto the vertex $(k)$ by iterated application of the deformation of a simplex onto a horn (see Propositions \ref{Axiom 4}, \ref{adjointSD}, and \ref{relcellcpx}, and Lemma \ref{predefretr}(2)). Hence, $\Delta^{p}$ is also contractible in $\dcal$.
		\end{rem}
		\begin{lem}\label{Kan}
			\begin{itemize}
				\item[$(1)$] For a diffeological space $X$, $S^{\dcal}X$ is a Kan complex.
				\item[$(2)$] If $f\simeq g  :X\longrightarrow Y$ in $\mathcal{D}$, then $S^{\mathcal{D}}f\simeq S^{\mathcal{D}}g  :S^{\mathcal{D}}X\longrightarrow S^{\mathcal{D}}Y$ in $\scal$. In particular, if $X \simeq Y$ in $\dcal$, then $S^{\dcal}X \simeq S^{\dcal}Y$ in $\scal$.
			\end{itemize}
		\end{lem}
	\begin{proof}
	(1) We show that, for any solid arrow diagram in $\scal$
	\begin{center}
		\begin{tikzcd}
			\Lambda_k[p]  \arrow[hook]{d}\arrow{r}& S^{\dcal}X \\
			\Delta[p]  \arrow[dotted]{ru},
		\end{tikzcd}
	\end{center}
	there exists a dotted arrow that makes the diagram commute. By Proposition \ref{adjointSD}, this extension problem is equivalent to the extension problem
	\begin{center}
		\begin{tikzcd}
			\left|\Lambda_k[p] \right|_{\mathcal{D}}\arrow[hook]{d}\arrow{r}& X \\
			\left|\Delta[p] \right|_{\mathcal{D}}  \arrow[dotted]{ru}
		\end{tikzcd}
	\end{center}
	in $\dcal$, which has a solution by Propositions \ref{relcellcpx} and Axiom 4 (Proposition \ref{Axiom 4}).\\
	$(2)$ Let $H  :X\times I\longrightarrow Y$ be a $\mathcal{D}$-homotopy connecting $f$ to $g$. Then, the composite
	$$S^{\mathcal{D}}X\times\Delta[1]\xrightarrow{1\times\iota}S^{\mathcal{D}}X\times S^{\mathcal{D}}I\cong S^{\mathcal{D}}(X\times I)\xrightarrow{S^{\mathcal{D}}H}S^{\mathcal{D}}Y$$
	is a simplicial homotopy connecting $S^{\mathcal{D}}f$ to $S^{\mathcal{D}}g$, where $\iota:\Delta[1]\longrightarrow S^{\mathcal{D}}I$ is a simplicial map corresponding to the canonical diffeomorphism $\Delta^1\longrightarrow I$.
\end{proof}
	Let $S :\czero\longrightarrow\scal$ be the topological singular functor, and $|\ | : \scal \longrightarrow\czero$ be the topological realization functor defined by $|K|=\underset{\Delta\downarrow K}{\colim}\ \Delta^n_{\text{top}}$ in $\czero$; see Section 1 for the simplex category $\Delta \downarrow K$. It is easily seen that $(|\ |, S)$ is an adjoint pair (cf. \cite[p. 7]{GJ}). Note that $\underset{\Delta\downarrow K}{\colim}\ \Delta^n_{\text{top}}$ in $\czero$ coincides with $\underset{\Delta\downarrow K}{\colim}\ \Delta^n_{\text{top}}$ in $\tcal$; in other words, $|K|$ coincides with the ordinary topological realization of $K$ (see Lemma \ref{adjointC0T}).
	\begin{lem}\label{compositeadjoint}
		The composite of the adjoint pairs
		\[
		|\ |_{\dcal}: \scal \rightleftarrows \dcal: S^{\dcal} \ \text{and} \ \widetilde{\cdot}: \dcal \rightleftarrows \czero: R
		\]
		is just the adjoint pair
		\[
		|\ |: \scal \rightleftarrows \czero: S.
		\]
		\begin{proof}
			
			For $X \in \ccal^{0}$,
			$
			\dcal(\Delta^{p}, RX) \cong \ccal^{0}(\Delta^{p}_{\mathrm{top}}, X)
			$
			(Propositions \ref{adjointDC0} and \ref{Axiom 1}), and therefore,
			$
			S^{\dcal}\circ R = S.
			$
			Since $\widetilde{|\ |_{\dcal}}$ and $|\ |$ are left adjoints to $S^{\dcal}\circ R$ and $S$, respectively, the equality $S^{\dcal}\circ R = S$ implies the equality $\widetilde{|\ |_{\dcal}} = |\ |.$
		\end{proof}
	\end{lem}
	\subsection{Proof of Theorem \ref{model}}
	In this subsection, we prove Theorem \ref{model}. 
	We need several lemmas for this.
	\par\indent
	Define the set $\jcal$ of morphisms of $\dcal$ by
	\begin{eqnarray*}
		\jcal & = & \{\Lambda^{p}_{k} \longhookrightarrow \Delta^{p} \ | \ p>0,\ 0 \leq k \leq p \}.
	\end{eqnarray*}
	\begin{lem}\label{F in D}
		Let $f  :X\longrightarrow Y$ be a morphism of $\mathcal{D}$.
		\begin{itemize}
			\item[$(1)$]
			The following conditions are equivalent:
			\begin{itemize}
				\item[$(\mathrm{i})$]
				$f  :X\longrightarrow Y$ is a fibration;
				\item[$(\mathrm{ii})$]
				$S^{\mathcal{D}} f : S^{\mathcal{D}} X\longrightarrow S^{\mathcal{D}} Y$ is a fibration;
				\item[$(\mathrm{iii})$]
				$f$ has the right lifting property with respect to $\jcal$.
			\end{itemize}
			\item[$(2)$]
			The following conditions are equivalent:
			\begin{itemize}
				\item[$(\mathrm{i})$]
				$f  :X\longrightarrow Y$ is both a fibration and a weak equivalence;
				\item[$(\mathrm{ii})$]
				$S^{\mathcal{D}} f : S^{\mathcal{D}} X\longrightarrow S^{\mathcal{D}} Y$ is both a fibration and a weak equivalence;
				\item[$(\mathrm{iii})$]
				$f$ has the right lifting property with respect to $\ical$.
			\end{itemize}
		\end{itemize}
		\begin{proof}$(1)$
			The equivalence of (i) and (iii) is obvious from the definition of fibration. The equivalence of (ii) and (iii) follows from Propositions \ref{adjointSD} and \ref{relcellcpx}; see the proof of Lemma \ref{Kan}(1).\\
			$(2)$
			(i)$\Leftrightarrow$(ii)
			The equivalence is obvious from part 1 and the definition of weak equivalence.\\
			(ii)$\Leftrightarrow$(iii)
			Recall that a simplicial map $\varphi  : K\longrightarrow L$ is both a fibration and a weak equivalence if and only if $\varphi  : K\longrightarrow L$ has the right lifting property with respect to the inclusions $\dot{\Delta}[n]\longhookrightarrow \Delta[n]$ (\cite[Lemma 2.4]{K}). Then, the equivalence follows from Propositions \ref{adjointSD} and \ref{relcellcpx}.
		\end{proof}
	\end{lem}
	\begin{lem}\label{defretr}
		Suppose given a diagram
		\begin{center}
			\begin{tikzcd}
				\underset{\lambda\in\Lambda}{\coprod}\Lambda^{p_\lambda}_{k_\lambda}\arrow{d}\arrow[hook]{r}& \underset{\lambda\in\Lambda}{\coprod}\Delta^{p_{\lambda}}\\
				X
			\end{tikzcd}
		\end{center}
		in $\mathcal{D}$. Then $X$ is a deformation retract of $X \underset{\underset{\lambda\in\Lambda}{\coprod}\Lambda^{p_\lambda}_{k_\lambda}}{\cup}\underset{\lambda\in\Lambda}{\coprod}\Delta^{p_{\lambda}}$ in $\dcal$.
	\end{lem}
	\begin{proof}
		The result follows immediately from Proposition \ref{Axiom 4} and Lemma \ref{predefretr}(2).	
	\end{proof}
	
	For a subset $A$ of $\Delta^p$, $A_{\text{top}}$ is the set $A$ endowed with the induced topology of $\Delta^p_{\text{top}}$.
	
	\begin{lem}\label{underlyinghornbdry}
		The underlying topological spaces of $\Lambda^p_k$ and $\dot{\Delta}^p$ are $\Lambda^p_{k\ \mathrm{top}}$ and $\dot{\Delta}^p_{\mathrm{top}}$ respectively.
		\begin{proof}
			The result follows immediately from Proposition \ref{relcellcpx} and Lemma \ref{compositeadjoint}.
		\end{proof}
	\end{lem}
	
	\begin{lem}\label{sequence}
		Let
		$$
		X_0\overset{i_1}{\longhookrightarrow}X_1\overset{i_2}{\longhookrightarrow}X_2\overset{i_3}{\longhookrightarrow}\cdots
		$$
		be a sequence in $\dcal$ satisfying one of the following conditions:
		\begin{itemize}
			\item[$\mathrm{(i)}$] Each $i_n$ fits into a pushout diagram of the form
			\begin{center}
				\begin{tikzcd}
					\underset{\lambda\in\Lambda_n}{\coprod}\dot{\Delta}^{p_\lambda}\arrow{d}\arrow[hook]{r}&\underset{\lambda\in\Lambda_n}{\coprod}\Delta^{p_\lambda}\arrow{d}\\
					X_{n-1}\arrow[hook]{r}{i_n} & X_n.
				\end{tikzcd}
			\end{center}
			\item[$\mathrm{(ii)}$] Each $i_n$ fits into a pushout diagram of the form
			\begin{center}
				\begin{tikzcd}
					\underset{\lambda\in\Lambda_n}{\coprod}\Lambda^{p_{\lambda}}_{k_{\lambda}}\arrow{d}\arrow[hook]{r}&\underset{\lambda\in\Lambda_n}{\coprod}\Delta^{p_\lambda}\arrow{d}\\
					X_{n-1}\arrow[hook]{r}{i_n} & X_n.
				\end{tikzcd}
			\end{center}
		\end{itemize}
		$Set$ $X = \underset{\rightarrow}{\lim} \ X_{n}$ and let $A$ be a diffeological space with $\widetilde{A}$ compact. Then, any smooth map $f: A \longrightarrow X$ factors through $X_{n}$ for some $n$.
	\end{lem}
	
	\begin{proof}
		Note that $\Lambda^{p}_{k} \longhookrightarrow \Delta^{p}$ is the composite of $\Lambda^{p}_{k} \longhookrightarrow \dot{\Delta}^{p}$ and $\dot{\Delta}^{p} \longhookrightarrow \Delta^{p}$, and that $\Lambda^{p}_{k} \longhookrightarrow \dot{\Delta}^{p}$ is the pushout of $\dot{\Delta}^{p-1} \longhookrightarrow \Delta^{p-1}$ (see Propositions \ref{adjointSD} and \ref{relcellcpx}, and the argument in the proof of Proposition 7.1). Then, we may assume that the sequence $\{X_{n} \}$ satisfies condition (i).
		\par\indent
		Define the set $\widetilde{\ical}$ of morphisms of $\czero$ by
		\[
		\widetilde{\ical} = \{\dot{\Delta}^{p}_{\mathrm{top}} \longhookrightarrow \Delta^{p}_{\mathrm{top}} \ | \ p \geq 0 \}.
		\]
		Since the canonical map $X_{0} \longhookrightarrow X$ is a sequential relative $\ical$-cell complex (see \cite[Definitions 15.1.1 and 15.1.2]{MP}), the map $\widetilde{X_{0}} \longhookrightarrow \widetilde{X}$ is a sequential relative $\widetilde{\ical}$-cell complex by Propositions \ref{adjointDC0} and \ref{Axiom 1}, and Lemma \ref{underlyinghornbdry}. Therefore, we can easily see that $f: A \longrightarrow X$ factors through some $X_{n}$ in $Set$ (cf. \cite[Lemma 8.7]{DS}), and hence in $\dcal$ (see the comment after Proposition \ref{category D}).
	\end{proof}

\begin{proof}[Proof of Theorem \ref{model}]
	Let us verify the following model axioms:
	\begin{itemize}
		\item[M1]
		$\mathcal{D}$ is closed under all small limits and colimits.
		\item[M2]
		If $f$ and $g$ are maps in $\mathcal{D}$ such that $gf$ is defined and if two of the three maps $f$, $g$, $gf$ are weak equivalences, then the third also is.
		
		\item[M3]
		If $f$ is a retract of $g$ and $g$ is a weak equivalence, fibration, or cofibration, then $f$ also is.
		\item[M4]
		Given a commutative solid arrow diagram
		$$
		\begin{tikzcd}
		A \arrow{d}{i} \arrow{r} & X \arrow{d}{p}\\
		B \arrow[dashed]{ur} \arrow{r}& Y
		\end{tikzcd}
		$$
		the dotted arrow exists, making the diagram commute, if either
		\begin{itemize}
			\item[(i)] $i$ is a cofibration and $p$ is a trivial fibration (i.e., a fibration that is also a weak equivalence), or
			\item[(ii)] $i$ is a trivial cofibration (i.e., a cofibration that is also a weak equivalence) and $p$ is a fibration.
		\end{itemize}

		\item[M5]
		Every map $f$ has two functorial factorizations:
		\begin{itemize}
			\item[(i)]
			$f$ = $pi$, where $i$ is a cofibration and $p$ is a trivial fibration, and
			\item[(ii)]
			$f$ = $qj$, where $j$ is a trivial cofibration and $q$ is a fibration. 
		\end{itemize}
	\end{itemize}
	
	\item[M1]
	is satisfied by Proposition \ref{category D}(1). M2 is obvious. M3 is not difficult (cf. \cite[2.7 and 8.10]{DS}).
	\item[M5]
	(ii) Applying the infinite gluing construction (IGC) for ${\jcal}=\{\Lambda^p_k\longhookrightarrow\Delta^p\ (p>0,0\leq k\leq p)\}$ to $f  :X\longrightarrow Y$, we obtain the factorization
	\begin{center}
		\begin{tikzcd}
			X\arrow[swap]{rd}{f}\arrow{r}{j_\infty}&G^{\infty}({\jcal},f)\arrow{d}{q_\infty}\\
			& Y
		\end{tikzcd}
	\end{center}
	(see \cite[pp. 104-105]{DS}). By Lemmas \ref{underlyinghornbdry} and \ref{sequence}, $q_{\infty}$ is a fibration. Since $j_\infty$ is obviously a cofibration by construction, we need to only see that $j_\infty$ is a weak equivalence.
	\par\indent
	By Lemma \ref{defretr}, the inclusion $X\longhookrightarrow G^n({\jcal},f)$ is a $\dcal$-homotopy equivalence, and hence, a weak equivalence (Lemma \ref{Kan}(2)). Thus, $j_\infty :X\longhookrightarrow G^\infty({\jcal},f)$ is also a weak equivalence by Lemma \ref{sequence}.
	\par\noindent (i)
	Applying the infinite gluing construction (IGC) for $\ical=\{\dot{\Delta}^p\longhookrightarrow{\Delta}^p\ (p\geq 0)\}$ to $f  :X\longrightarrow Y$, we obtain the factorization
	\begin{center}
		\begin{tikzcd}
			X\arrow[swap]{rd}{f}\arrow{r}{i_\infty}&G^{\infty}({\ical},f)\arrow{d}{p_\infty}\\
			& Y
		\end{tikzcd}
	\end{center}
	(see \cite[pp. 104-105]{DS}). By Lemmas \ref{underlyinghornbdry} and \ref{sequence}, $p_{\infty}$ is a trivial fibration (see Lemma \ref{F in D}(2)). By construction, $i_{\infty}$ is obviously a cofibration.\item[M4]
	can also be shown by an argument similar to that in the case of the category of topological spaces (\cite[p. 110]{DS}).\\
	\par
	By Lemma \ref{F in D}, $\dcal$ is a cofibrantly generated model category (\cite[Definition 15.2.1]{MP}); the sets $\ical$ and $\jcal$ are the generating cofibrations and generating trivial cofibrations. By Lemmas \ref{underlyinghornbdry} and \ref{sequence}, the model structure on $\dcal$ is compactly generated.		
	By Lemmas \ref{Kan}(1) and \ref{F in D}(1), every diffeological space is fibrant.  
\end{proof}	
\begin{rem}\label{9.7}
	We can also see that $\dcal$ is a cellular model category; conditions (1) and (2) in \cite[Definition 12.1.1]{Hi} are easily checked using Propositions \ref{adjointDC0} and \ref{Axiom 1}, and Lemma \ref{underlyinghornbdry} (cf. the proof of Lemma \ref{sequence}), and condition (3) is checked using \cite[Proposition 10.9.6 and Example 10.9.3]{Hi} and Proposition \ref{category D}(1). 
\end{rem}

\subsection{Proof of Theorem \ref{homotopygp}}
In this subsection, we prove Theorem \ref{homotopygp}, using results of \cite{CW}, and then characterize weak equivalences of diffeological spaces in terms of smooth homotopy groups.
\par\indent
For the proof of Theorem \ref{homotopygp}, we need a few lemmas.

\begin{lem}\label{bdydeform}
	For a fixed real number 
	$\epsilon \in (0, \frac{1}{p+1}), \partial_{\epsilon}\Delta^{p}$
	is the diffeological subspace of $\Delta^p$ defined by
	$$
	\partial_\epsilon \Delta^{p} = \{(x_0, \ldots, x_p) \in \Delta^p \ |\ x_i\leq \epsilon \ \text{for some} \ i \}.
	$$
	Then, there exists a $\dcal$-homotopy 
	$T  : \Delta^p \times I \longrightarrow \Delta^p$ such that 
	$T(\cdot, 0) = 1_{\Delta^p}$ and $T\ |_{\partial_{\epsilon}\Delta^{p} \times I}$ is a $\dcal$-deformation of $\partial_{\epsilon}\Delta^p$ onto $\dot{\Delta}^p$.
	\begin{proof}
		There exists a $\dcal$-deformation 
		$D: (\Delta^p - \{b_p\})\times I \longrightarrow \Delta^p-\{b_p \}$
		of $\Delta^p - \{b_p\}$ onto $\dot{\Delta}^p$ (Proposition \ref{deformontbdry}); from the construction, we can choose a $\dcal$-deformation $D$ such that $D$ restricts to a $\dcal$-deformation of $\partial_{\epsilon}\Delta^{p}$ onto $\dot{\Delta}^{p}$. Choose a smooth function
		$\rho  : \Delta^p \longrightarrow [0, 1]$ such that $\rho \equiv 0$ near $b_p$ and $\rho \equiv 1$ near $\partial_{\epsilon}\Delta^p$ (see Lemma \ref{firstproperty}). Then, define the map
		$T  : \Delta^p \times I \longrightarrow \Delta^{p}$
		by
		\[
		T(x, t)=
		\begin{cases}
		D(x, \rho(x)t) & \text{for} \ \ x \neq b_p \\
		b_p & \text{for} \ \ x = b_p.
		\end{cases}
		\]
		It is easily seen that $T$ is the desired $\dcal$-homotopy.
	\end{proof}
\end{lem}
A {\sl diffeological pair} $(X, A)$ consists of a diffeological space $X$ and a diffeological subspace $A$ of $X$. We can define the {\sl $\dcal$-homotopy set} 
$[(X, A),(Y, B)]_{\dcal}$
between diffeological pairs $(X, A)$ and $(Y, B)$ using the unit interval $I$ similar to the case of topological pairs (cf. Section 2.4), which is the same as for the case defined using the line $\rbb$ (\cite[p. 1277]{CW}); see also Remark \ref{Dhomotopy}.

\begin{lem}\label{smoothhomotopygp}
	Let $(X, x)$ be a pointed diffeological space. For a fixed real number 
	$\epsilon \in (0, \frac{1}{p+1})$,
	there exists a natural bijection
	$$
	\pi^{\dcal}_{p}(X, x) \cong [(\Delta^{p}, \partial_{\epsilon}\Delta^{p}), (X, x)]_{\dcal}.	
	$$
	\begin{proof}
		Let $\partial_{\epsilon}\Delta^{p}_{\mathrm{\text{\text{sub}}}}$ denote the subset $\partial_{\epsilon}\Delta^p$ endowed with the sub-diffeology of $\Delta^{p}_{\mathrm{\text{\text{sub}}}}$. By \cite[Theorem 3.2]{CW}, $\pi^{\dcal}_{p}(X, x)$ is naturally isomorphic to $[(\Delta^{p}_{\mathrm{\text{\text{sub}}}}, \partial_{\epsilon}\Delta^{p}_{\mathrm{\text{\text{sub}}}}), (X, x)]_{\dcal}$, which is obviously isomorphic to $[(\Delta^p, \partial_{\epsilon}\Delta^p), (X, x)]_{\dcal}$ (see Lemmas \ref{firstproperty} and \ref{secondproperty}).
	\end{proof}
\end{lem}
\begin{proof}[Proof of Theorem \ref{homotopygp}.]
	By Lemma \ref{smoothhomotopygp}, we identify $\pi^{\dcal}_{p}(X, x)$ with
	$[(\Delta^{p}, \partial_{\epsilon}\Delta^{p}), (X, x)]_{\dcal}$
	for a fixed $\epsilon \in (0, \frac{1}{p+1})$. 
	Since an affine map (i.e., a map preserving convex combinations) from $\Delta^{r}$ to $\Delta^{p} \times I$ is smooth by Axiom 2 (Proposition \ref{Axiom 2}), we have the natural map
	$$
	\varTheta_{X}:\pi^{\dcal}_{p}(X, x) \longrightarrow \pi_{p}(S^{\dcal}X, x).
	$$
	$\varTheta_{X}$ is surjective by Lemma \ref{bdydeform}.
	\par\indent
	Let us see that $\varTheta_{X}$ is injective. Suppose that 
	$\varTheta_{X}([f]) = \varTheta_{X}([g])$.
	Then, there is a smooth function 
	$F :\Delta^{p+1} \longrightarrow X$
	such that 
	\[
	F \circ d^i=
	\begin{cases}
	f & \text{if} \ i = p, \\
	g & \text{if} \ i = p+1, \\
	0 & \text{if} \ i \neq p, p+1,	
	\end{cases}
	\]
	where $0$ denotes the constant map to the base point. (See Section 1.2 for the definition of $d^i :\Delta^p \longrightarrow \Delta^{p+1}$.) By applying Lemma \ref{bdydeform} to $\Delta^{p+1}$, we can assume that $F = 0$ near $\mathrm{sk}_{p-1}\ \Delta^{p+1}$. Thus, $F$ defines a smooth map $\Delta^{p+1}_{\text{\text{sub}}}$ to $X$ (Lemma \ref{secondproperty}). Hence, we obtain the $\dcal$-homotopy relative to $\partial_{\epsilon'}\Delta^p$
	$$
	\Delta^p \times I \xrightarrow{id \times 1} \Delta^{p}_{\mathrm{\text{\text{sub}}}} \times I \xrightarrow{\ \ \beta\ \ } \Delta^{p+1}_{\text{\text{sub}}} \xrightarrow{\ \ F\ \ } X
	$$
	connecting $f$ to $g$, where $\epsilon'$ is a sufficiently small positive number and $\beta$ is a smooth map defined by 
	$$
	\beta(x_{0}, \ldots, \ x_p, \ t) = (x_0, \ldots, \ x_{p-1}, \ tx_p, \ (1-t)x_p).
	$$
	\par\indent
	Last, let us show that $\varTheta_{X}$ is a group homomorphism for $p > 0$. Define the map
	$$ 
	\gamma : \Delta^{p+1} \longrightarrow \Delta^{p}_{(p-1)} \underset{\langle 0, \ldots, p-2, p \rangle}{\cup}\Delta^{p}_{(p+1)}
	$$
	by
	\[
	\gamma(x_0, \ldots, x_{p+1})=
	\begin{cases}
	(x_0, \ldots, x_{p-2}, 0, x_p + 2x_{p-1}, x_{p+1}-x_{p-1}) & \text{if} \ x_{p+1} \geq x_{p-1} \\
	(x_0, \ldots, x_{p-2}, x_{p-1} - x_{p+1}, x_p + 2x_{p+1}, 0) & \text{if} \ x_{p+1} \leq x_{p-1}.  
	\end{cases}
	\]
	For smooth maps
	$f,g: (\Delta^{p}, \partial_{\epsilon}\Delta^{p}) \longrightarrow (X, x)$, the composite 
	$$
	\Delta^{p+1} \xrightarrow{\ \ \gamma\ \ } \Delta^{p}_{(p-1)} \underset{\langle 0, \ldots, p-2, p \rangle}{\cup}\Delta^{p}_{(p+1)} \xrightarrow{f+g} X
	$$
	is smooth, though $\gamma$ is not smooth. Thus, the product $\varTheta_{X}([f])\cdot\varTheta_{X}([g])$ 
	is defined to be the composite 
	$$
	\Delta^p \xrightarrow{\ \ d^p\ \ } \Delta^{p+1} \xrightarrow{(f+g)\circ \gamma} X.
	$$ 
	This coincides with $\varTheta_{X}([f]\cdot[g])$ (see the argument in the proof of \cite[Theorem 4.11]{CW}).
\end{proof}
\begin{cor}\label{characterization}
	Let 
	$f  : X \longrightarrow Y$ be a smooth map between diffeological spaces. Then, $f$ is a weak equivalence if and only if 
	$$
	\pi^{\dcal}_{p}(f):\pi^{\dcal}_{p}(X, x) \longrightarrow \pi^{\dcal}_{p}(Y, f(x))
	$$
	is bijective for any $p \geq 0$ and any $x \in X$.
\end{cor}
\begin{proof}
	The result follows immediately from the definition of a weak equivalence in $\dcal$ and Theorem \ref{homotopygp}.
\end{proof}
\appendix
\section{Diffeology of $\Delta^{p}_{\mathrm{sub}}$}
In this appendix, we see that $\Delta^{p}_{\mathrm{sub}}$ satisfies neither Axiom 3 nor 4 for $p > 1$, which is why we have newly introduced the standard $p$-simplices $\Delta^{p}$.
\par\indent
The following lemma shows that the set $\{\Delta^{p}_{\mathrm{sub}} \}_{p \geq 0}$ satisfies Axioms 1 and 2.
\begin{lem}\label{affirmative}
	$(1)$ The underlying topological space of $\Delta^{p}_{\mathrm{sub}}$ is the topological standard $p$-simplex for $p \geq 0$.\\
	$(2)$ Any affine map $f:\Delta^{p}_{\mathrm{sub}} \longrightarrow \Delta^{q}_{\mathrm{sub}}$ is smooth.
	\begin{proof}
		(1) The result follows from Lemma \ref{underlyingconvex}.\\
		(2) Obvious.
	\end{proof}	
\end{lem}
Using Lemma \ref{affirmative}(2), we can define the functor
$$
|\ |'_{\dcal} : \scal \longrightarrow \dcal
$$
by $|K|'_{\dcal} = \underset{\Delta\downarrow K}{\mathrm{colim}}\ \Delta^{n}_{\mathrm{sub}}$ (cf. Section 1.2). The following proposition shows that $\Delta^{p}_{\mathrm{sub}}$ satisfies neither Axiom 3 nor 4 for $p > 1$. Recall the definition of the horn $\Lambda^{p}_{k}$, and note that $\Lambda^{p}_{k\ \mathrm{sub}}$ is a diffeological subspace of $\Delta^{p}_{\mathrm{sub}}$  (Section 1.2).
\begin{prop}\label{counterex}
	$(1)$ The canonical smooth injection
	$$
	|\dot{\Delta}[p]|'_{\dcal} \longhookrightarrow \Delta^{p}_{\mathrm{sub}}
	$$
	is not a $\dcal$-embedding for $p > 1$.\\
	$(2)$	$\Lambda^{p}_{k\ \mathrm{sub}}$ is not a retract of $\Delta^p_{\mathrm{sub}}$ in $\dcal$ for $p>1$.
	\begin{proof}
		(1) Identify $|\dot{\Delta}[p]|'_{\dcal}$ with the boundary $\dot{\Delta}^{p}$ set-theoretically via the canonical injection. We show that $id :|\dot{\Delta}[p]|'_{\dcal} \longrightarrow \dot{\Delta}^{p}_{\mathrm{sub}}$ is not a diffeomorphism.
		We can easily construct a smooth injection
		$$
		c: (-1, 1) \longrightarrow \dot{\Delta}^{p}_{\mathrm{sub}}
		$$
		such that
		\begin{itemize}
			\item[$\bullet$] $c((-1, 0))$ is contained in the open simplex $(0, \ldots, p-1, \hat{p})$.
			\item[$\bullet$] $c((0, 1))$ is contained in the open simplex $(0, \ldots, \widehat{p-1}, p)$.
		\end{itemize}
		(cf. \cite[Excercise 15 2)]{I}.) However, the injection
		$
		c :(-1, 1) \longrightarrow |\dot{\Delta}[p]|'_{\dcal}
		$
		is not smooth since $c :(-\epsilon, \epsilon) \longrightarrow |\dot{\Delta}[p]|^{'}_{\dcal}$ does not factor through 
		$
		|\sigma|^{'}_{\dcal} : \Delta^{n}_{\mathrm{sub}} \longrightarrow |\dot{\Delta}[p]|^{'}_{\dcal}
		$
		for any $\epsilon > 0$ and any $\sigma \in \Delta\downarrow \dot{\Delta}[p]$ (see the proof of Proposition \ref{category D}(1) and \cite[p. 90]{CSW} for the diffeology of $|K|'_{\dcal} = \underset{\Delta\downarrow K}{\text{colim}} \ \Delta^{n}_{\mathrm{sub}}$).\\
		(2)
		Suppose that there is a retraction $r  : \Delta^p_{\text{sub}}\longrightarrow \Lambda^p_{k\ \text{sub}}$ in $\mathcal{D}$. Consider the solid arrow diagram
		\begin{equation*}
		\begin{tikzcd}
		\Delta_{\text{sub}}^p \arrow{r}{r} \arrow[hook]{d} & 
		\Lambda^p_{k\ \text{sub}}  \arrow[hook]{d}\\
		\mathbb{A}^p  \arrow[dotted]{r}{\overline{r}}
		& \mathbb{A}^p
		\end{tikzcd}
		\end{equation*}
		of diffeological spaces. By Remark \ref{Frolicher} below, there is a smooth map $\overline{r}$ defined on an open neighborhood of the vertex $(k)$ in $\mathbb{A}^p$, which makes the above diagram commutative near $(k)$. Since the Jacobi matrix of $\overline{r}$ at $(k)$ is the identity matrix, $\overline{r}$ is a local diffeomorphism by the inverse function theorem, which is a contradiction.
	\end{proof}
\end{prop}
\begin{rem}\label{Frolicher}
	Set $K = \{(x_{1}, \ldots, x_{n}) \in \rbb^{n}\ |\ x_{1} \geq 0, \ldots, x_{n} \geq 0 \}$. Proposition 24.10 in \cite{KM} shows that any smooth map $f$ from the Fr\"{o}licher subspace $K$ of $\rbb^{n}$ to $\rbb^{m}$ extends to a smooth map from $\rbb^{n}$ to $\rbb^{m}$ (see \cite[Section 23]{KM} for Fr\"{o}licher spaces). Thus, in order to use this result in the diffeological context (the proof of Proposition A2(2)), we need the following observation which follows from Boman's theorem (\cite[Corollary 3.14]{KM}). Let $f$ be a set-theoretic map from a subset $A$ of $\rbb^{n}$ to $\rbb^{m}$. Then, the following conditions are equivalent:
	\begin{description}
	\item[\ \, \, $\mathrm{(i)}$] $f$ is a smooth map from the Fr\"{o}licher subspace $A$ to $\rbb^{m}$.
	\item[\ \ \ $\mathrm{(ii)}$] $f$ is a smooth map from the diffeological subspace $A$ to $\rbb^{m}$.
	\end{description}
\end{rem}

\begin{rem}\label{skeleton}
	By the argument in the proof of Proposition \ref{counterex}(2), we also see that $(\Lambda^{p}_{k}-\dot{\Delta}^{p-1}_{(k)})_{\mathrm{sub}}$ is not a retract of $\Delta^{p}_{\hat{k}\ \mathrm{sub}}$ for $p > 1$ (see Definition \ref{subsets} for the definitions of $\dot{\Delta}^{p-1}_{(k)}$ and $\Delta^{p}_{\hat{k}}$). On the other hand, $\Lambda^{p}_{k} - \dot{\Delta}^{p-1}_{(k)}$ is a retract of $\Delta^{p}_{\hat{k}}$ in $\dcal$ (Lemma \ref{key}). Noticing that $proj: A \times B \longrightarrow B$ is a $\dcal$-quotient map for $A, B \in \dcal$ with $A$ nonempty, we see that $\mathrm{sk}_{p-2} \ \Delta^p$ is just the set of $x \in \Delta^p$ such that the restriction of $id  : \Delta^{p} \longrightarrow \Delta^{p}_{\mathrm{sub}}$ to any open neighborhood of $x$ is not a diffeomorphism (see Lemmas \ref{firstproperty} and \ref{secondproperty}, Proposition \ref{goodnbd}, and Remark \ref{gn}(2)).
\end{rem}

\begin{rem}\label{fibrantapprox}
	As in \cite{H}, we can define the singular complex $S^{\dcal}_{\mathrm{sub}}\,X$ of a diffeological space $X$ using $\{\Delta^{p}_{\mathrm{sub}}\}_{p \geq 0}$ instead of $\{\Delta^{p} \}_{p \geq 0}$ (see Section 1.3 and Lemma \ref{affirmative}(2)). By Lemma \ref{firstproperty}, we have the natural inclusion $j_{X}:S^{\dcal}_{\mathrm{sub}}\,X \longhookrightarrow S^{\dcal}X$. In the succeeding paper, it is shown that $j_{X}$ is a fibrant approximation of $S^{\dcal}_{\mathrm{sub}}\,X$; see \cite{Quillenequiv} for the fact that $j_{X}$ induces isomorphisms on the homology.
\end{rem}



\begin{thebibliography}{}
	\bibitem{BH}
	J. Baez and A. Hoffnung, {\em Convenient categories of smooth
		spaces}, Transactions of the American Mathematical Society \textbf{363} (2011), no. 11, 5789-5825.
	
	
	
	
	
	\bibitem{CSW}
	J. D. Christensen, G. Sinnamon and E. Wu, {\em The D-topology
		for diffeological spaces,} Pacific Journal of Mathematics \textbf{272} (2014), no. 1, 87-110.
	
	\bibitem{CW}
	J. D. Christensen and E. Wu, {\em The homotopy theory of diffeological
		spaces,} New York J. Math \textbf{20} (2014) 1269-1303.
	
	
	
	
	
	\bibitem{DS}
	W. G. Dwyer and J. Spalinski, {\em Homotopy theories and model
		categories,} Handbook of algebraic topology (1995) 73-126.
	
	
	\bibitem{FK}
	A. Fr\"{o}licher and A. Kriegl, {\em Linear spaces and differentiation theory}, Vol. 13, John Wiley and Sons Inc, (1988).
	
	
	\bibitem{GJ}
	P. G. Goerss and J. F. Jardine, {\em Simplicial Homotopy, Theory}, Birkh\"{a}user,
	Verlag, Basel (1999).
	
	
	\bibitem{HS}
	T. Haraguchi and K. Shimakawa, {\em A model structure on the
		category of diffeological spaces,} preprint, http://arxiv.org/abs/1311.5668.
	
	\bibitem{H}
	G. Hector, {\em G\'{e}om\'{e}trie et topologie des espaces diff\'{e}ologiques,}
	Analysis and geometry in foliated manifolds (Santiago de Compostela, 1994)
	(1995) 55-80.
	
	\bibitem{Hi}
	P.  S. Hirschhorn, {\em Model categories and their localizations}, No. 99,
	American Mathematical Soc, (2009).
	
	\bibitem{Ho}
	M. Hovey, {\em Model categories}, No. 63, American Mathematical Soc,
	(2007).
	
	\bibitem{I}
	P. Iglesias-Zemmour, {\em Diffeology,} Vol. 185, American Mathematical
	Soc, (2013).
	
	
	
	\bibitem{K}
	H. Kihara, {\em Minimal fibrations and the organizing theorem of
		simplicial homotopy theory,} Ricerche di Matematica \textbf{63} (2014), no. 1, 79-91.
	
	\bibitem{Quillenequiv}
	H. Kihara, {\em Quillen equivalences between the model categories of smooth spaces}, {\em simplicial sets, and arc-generated spaces}, arXiv preprint arXiv:1702.04070 (2017). 
	
	\bibitem{KM}
	A. Kriegl and P. W. Michor, {\em The convenient setting of global
		analysis,} Vol. 53, American Mathematical Society (1997).
	
	\bibitem{L}
	F. L\'{a}russon, {\em Model structures and the Oka principle}, Journal of Pure and Applied Algebra \textbf{192} (2004), no. 1, 203-223.
	
	\bibitem{Mac}
	{S. Mac Lane,} {\em Categories for the working mathematician,} Graduate
	Texts in Mathematics 5 1998
	
	\bibitem{MP}
	J. P. May and K. Ponto, {\em More concise algebraic topology:
		localization, completion, and model categories,} University of Chicago
	Press, (2011).
	
	\bibitem{MV}
	F. Morel and V. Voevodsky, {\em $\mathrm{A}^{1}$-homotopy theory of schemes},
	Publications Mathématiques de l'IHES \textbf{90} (1999), no. 1, 45-143.
	
	\bibitem{O}
	P. A. {\O}stv{\ae}r, {\em Homotopy theory of C*-algebras}, Springer Science \&
	Business Media, (2010).
	\bibitem{Q1}
	D. G. Quillen, {\em Homotopical algebra}, Vol. 43 of Lecture Notes in
	Mathematics, (1967).
	
	
	
	\bibitem{SYH}
	K. Shimakawa, K. Yoshida and T. Haraguchi, {\em Homology and cohomology via enriched bifunctors,} preprint, http://arxiv.org/abs/1010.3336.
	
	
	
	
	
	
	
	
	\bibitem{Wu}
	E. Wu, {\em A homotopy theory for diffeological spaces,} Diss. The University
	of Western Ontario London, (2012).
	
	
	\bibitem{Wy}
	O. Wyler, {\em Convenient categories for topology,} General Topology and
	its Applications \textbf{3} (1973), no. 3, 225-242.

\end{thebibliography}


\end{document}